\documentclass[reqno,11pt]{amsart}
\usepackage{amsmath,amssymb,amsthm, verbatim}
\usepackage{url}
\usepackage[usenames, dvipsnames]{color}
\usepackage{calrsfs}  

\usepackage[letterpaper,hmargin=1.5in,vmargin=1in]{geometry}

\usepackage{graphicx}
\usepackage{enumerate,pinlabel}
\usepackage{mathrsfs,graphicx,url}
\usepackage[usenames,dvipsnames]{xcolor}
\usepackage[colorlinks=true,linkcolor=red,citecolor=blue, urlcolor=Black]{hyperref}

\numberwithin{equation}{section}

\theoremstyle{plain}
\newtheorem{theorem}{Theorem}
\newtheorem*{theorem*}{Theorem}
\newtheorem*{lemma*}{Lemma}
\newtheorem{lemma}{Lemma}[section]
\newtheorem{proposition}[lemma]{Proposition}
\newtheorem{corollary}[lemma]{Corollary}

\theoremstyle{definition}

\theoremstyle{remark}

\newtheorem{remark}[lemma]{Remark}

\newcommand{\NN}{\mathbb{N}}
\newcommand{\RR}{\mathbb{R}}
\newcommand{\R}{\mathbb{R}}
\newcommand{\US}{\mathbb{S}}

\newcommand{\sca}{\mathrm{sc}}
\newcommand{\Psisc}{\Psi_{\mathrm{sc}}}
\newcommand{\lra}{\longrightarrow}

\newcommand{\Rot}{\mathrm{Rot}}
\newcommand\Ell{\mathop{\rm Ell}}

\DeclareMathOperator{\WF}{WF}
\newcommand\supp{\mathop{\rm supp}}

\newcommand*{\defeq}{\mathrel{\vcenter{\baselineskip0.5ex \lineskiplimit0pt

                     \hbox{\scriptsize.}\hbox{\scriptsize.}}}%
                     =}
\newcommand{\vl}{\mathsf{l}}      
\newcommand{\Rp}{\mathcal{R}_+}
\newcommand{\Rm}{\mathcal{R}_-}
\newcommand\utilde{\tilde u}
\newcommand\vol{\mathrm{vol}}

\newcommand\Id{\operatorname{Id}}
\newcommand{\sw}{\mathsf{l}}   
\newcommand\X{\mathcal{X}}
\newcommand\Y{\mathcal{Y}}

\newcommand{\wmac}{F}

\usepackage{color}
\definecolor{darkgreen}{cmyk}{1,0,1,.2}
\definecolor{m}{rgb}{1,0.1,1}
\definecolor{b}{rgb}{0,0.1,1}

\begin{document}
\title[Regularity of the nonlinear scattering matrix]{Regularity of the Scattering Matrix for Nonlinear Helmholtz Eigenfunctions}

\author{Jesse Gell-Redman}
\address{School of Mathematics and Statistics, University of Melbourne, Melbourne, Victoria, Australia}
\email{jgell@unimelb.edu.au}
\author{Andrew Hassell}
\address{Mathematical Sciences Institute, Australian National University, Acton, ACT 2601, Australia}
\email{Andrew.Hassell@anu.edu.au}
\author{Jacob Shapiro}
\address{Department of Mathematics, University of Dayton, Dayton, OH 45469-2316}
\email{jshapiro1@udayton.edu}

\thanks{The authors were supported in part by the Australian Research Council through grant DP180100589}

\begin{abstract}
 We study the nonlinear Helmholtz equation $(\Delta - \lambda^2)u = \pm |u|^{p-1}u$ on $\R^n$, $\lambda > 0$, $p \in \NN$ odd, and more generally $(\Delta_g + V - \lambda^2)u = N[u]$, where $\Delta_g$ is the (positive) Laplace-Beltrami operator on an asymptotically Euclidean or conic manifold, $V$ is a short range potential, and $N[u]$ is a more general polynomial nonlinearity. Under the conditions $(p-1)(n-1)/2 > 2$ and $k > (n-1)/2$,  for every $f \in H^k(\US^{n-1}_\omega)$ of sufficiently small norm, we show there is a nonlinear Helmholtz eigenfunction taking the form
\begin{equation*} 
u(r, \omega) = r^{-(n-1)/2} \Big( e^{-i\lambda r} f(\omega) + e^{+i\lambda r}
b(\omega) + O(r^{-\epsilon}) \Big), \qquad \text{as } r \to \infty, 
\end{equation*} 
for some $b \in H^k(\US_\omega^{n-1})$ and $\epsilon > 0$. That is, the nonlinear scattering matrix $f \mapsto b$ preserves Sobolev regularity,  which is an improvement over the authors' previous work \cite{ghsz20} with Zhang, that proved a similar result with a loss of four derivatives. 
\end{abstract}

\maketitle

\section{Introduction and statement of results}

We consider a Hamiltonian $H = \Delta_g + V$ defined on
$\mathbb{R}^n$, where $g$ is an asymptotically Euclidean Riemannian
metric in the sense defined below (an example is any smooth, compactly supported perturbation of the flat metric), and $V \in C^\infty(\mathbb{R}^n)$ is a real valued potential
function which is short range and satisfies symbolic estimates in the sense that
$$
| D_z^\alpha V(z) | \le C \langle z \rangle^{-\gamma - |\alpha|}
$$
for some $\gamma > 1$. 
We study the scattering problem for the nonlinear Helmholtz equation
\begin{equation}\label{eq:equation}
(H- \lambda^2) u = N[u], 
\end{equation}
for certain polynomial nonlinearities $N$. The admissible nonlinearities are defined below, but for now we note that examples include
$$
N[u] = \big(c_1 |u|^{p-1} + c_2 |\nabla u|^{p-1} \big) u, 
$$
where $p \geq 3$ is an odd integer. For prescribed, sufficiently small  data $f \in
H^k(\mathbb{S}^{n-1})$, we seek $u$ solving equation \eqref{eq:equation}, such that 
\begin{equation}\label{eq:asymptotic}
u(r, y) \sim r^{-(n-1)/2} \Big( e^{-i \lambda r}f(y) + e^{+i \lambda r}
b(y) \Big), \qquad b \in H^{k'}(\mathbb{S}^{n-1}). 
\end{equation}
Here $f$ is the ``incoming'' data and
$b$ is the ``outgoing'' data.  We refer to the association $f \mapsto
b$ as the ``nonlinear scattering matrix." In the linear setting, where $N[u] \equiv 0$, the map
$f \mapsto b$ is called the linear scattering matrix and denoted $S_{\mathrm{lin}}(\lambda)$. 
In this setting, it is a pseudodifferential operator of order zero on $\mathbb{S}^{n-1}$ composed with the antipodal map.

The results from \cite{ghsz20} include the following.
\begin{theorem*}\label{thm:old theorem}
 Assume that $(p-1)(n-1)/2 > 2$ and suppose that $k - 4 > (n-1) / 2$, $k \in \NN$.  Then there is $c > 0$ sufficiently
  small, such that for all $f \in H^{k}(\mathbb{S}^{n-1})$ with $\| f
  \|_{H^{k}} < c$, there is a solution $u$ to
  \eqref{eq:equation} satisfying \eqref{eq:asymptotic} with $k' = k - 4$.
\end{theorem*}

The purpose of this article is to prove, as is expected from the inherent
symmetry in the determination of $b$ from $f$ and vice-versa, that the
value of $k'$ can be taken equal to $k$, i.e. that the nonlinear scattering
matrix preserves Sobolev regularity.

\begin{theorem}
 Assume that $p$ is an odd integer and $k$ is an integer satisfying 
 \begin{equation} (p-1)\frac{n-1}{2} > 2 \text{ and } k > \max(1, \frac{n-1}{2}).
 \label{pkcond}\end{equation}
 There is an $c > 0$ such that for all  $f
\in H^{k}(\mathbb{S}^{n-1})$  with $\| f
  \|_{H^k} < c$, there is a solution $u$ to
  \eqref{eq:equation} satisfying \eqref{eq:asymptotic} with $b  \in
  H^k(\mathbb{S}^{n-1})$.

Moreover, $u$ is unique in the sense of the main theorem of
\cite{ghsz20}, described in detail in Section 4, and the error term in
the asymptotic expansion \eqref{eq:asymptotic}, 
$$
E_r : = u - r^{-(n-1)/2} \Big( e^{-i \lambda r}f + e^{+i \lambda r} b \Big) $$
satisfies
\begin{equation} \label{topology of convergence}
 \| E_r \|_{H^{k - 2}(\mathbb{S}^{n-1})} =  O(r^{-(n-1)/2 - \epsilon}) \qquad \text{for some $\epsilon > 0$}.
\end{equation}

Assume further that the stricter inequality $(p - 1)(n -1)/2 > 3$ holds. Then we
have a decomposition $b = S_{\mathrm{lin}}(\lambda) f
+ b_1$ where $S_{\mathrm{lin}}(\lambda)$ is as defined above and $b_1 \in H^{k
  + 1}$.  Moreover, still for this stricter condition, for $j \in
\mathbb{N}$, if $f \in H^{k + j}(\mathbb{S}^{n-1})$ (in addition to
the smallness condition in $H^k$) then $b_1 \in H^{k + j + 1}(\mathbb{S}^{n-1})$, in particular, $f \in C^\infty(\mathbb{S}^{n-1}) \implies b \in C^\infty(\mathbb{S}^{n-1})$.

\end{theorem}

To elaborate on the uniqueness statement, we show that, for an
appropriate microlocalizing pseudodifferential operator $A_-$ (see \eqref{eq:main commutator}) to the
incoming radial set, with $u_- = A_- u_0$ and $u_0 = P(\lambda) f$ the
linear generalized eigenfunction with incoming data $f$, then $u -
u_-$ is uniquely determined in a small ball around the origin in the
Hilbert space $H_+^{s, - 1/2 - \delta; 1, k - 1}$ defined in Section \ref{sec:resolvent}.

Note the convergence in \eqref{topology of convergence} is in $H^{k-2}(\US^{n-1})$.  This reflects the well known phenomenon from the linear
setting whereby an asymptotic expansion for an incoming (or outgoing)
approximate eigenfunction is produced by computing successive terms
in a formal expansion in negative powers of $r$.  This process in general produces a
``distributional'' expansion, in which coefficients of higher
order terms have decreasing regularity.  For example, in flat Euclidean space, an incoming approximate
eigenfunction $(\Delta_0 - \lambda^2)u_- \in
\mathcal{S}(\mathbb{R}^n)$ with incoming data $f \in C^\infty(\US^{n-1})$
admits an asymptotic expansion
\begin{equation*}\begin{gathered}
u_- \sim r^{-(n-1)/2} e^{i\lambda r} \sum_{j=0}^\infty r^{-j} v_j(y),  \\
v_0 = f, \quad v_{j+1} =  \frac{1}{2i (j + 1) \lambda} \left(
\Delta_{\mathbb{S}^{n-1}} + \frac{(n-1)(n-3)}{4} - j(j+1)\right)
v_j. 
\end{gathered}\end{equation*}
 From this one sees immediately that with $f \in H^k$ one can
obtain a partial expansion of an approximate eigenfunction in which
each subsequent term has a coefficient two orders rougher than the
previous one.  It therefore seems very natural that the convergence in our
theorem takes place in $H^{k-2}(\mathbb{S}^{n-1})$. See also Remark~\ref{rem:conv}. 



Our methods extend to prove a generalization of this result in the
setting of
asymptotically conic manifolds.  These are Riemannian manifolds
$(M^\circ, g)$ where $M^\circ$ is the interior of a compact manifold $M$ with
boundary $\partial M$ and $g$ is a so-called scattering metric, meaning it
takes the form 
\begin{equation}\label{ac-metric}
g = \frac{dx^2}{x^4} + \frac{h(x, y, dy)}{x^2},
\end{equation}
in a neighborhood of $\partial M$ where $x$ is a boundary defining function, i.e.\ $\partial M = \{ x
= 0 \}$ and $x \ge 0$ has that $dx$ is nonvanishing over $\partial M$,
and $y$ are coordinates on $\partial M$.  Here $h$ is a smooth $(0,2)$-tensor that restricts to a metric on
$\partial M$. Flat Euclidean space
is an example of an asymptotically conic space; write $M^\circ =
\mathbb{R}^n$ and include $\mathbb{R}^n \hookrightarrow
\overline{\mathbb{R}^n} =  \{ w \in \mathbb{R}^n :
|w|\le 1 \} =:M$ where the inclusion can be realized by the map $z
\mapsto z / (1 + \langle z \rangle)$ and note that the metric form is
realized by writing the flat
metric in polar coordinates and setting $x = 1/r$.

In the Euclidean case, $\partial M$ is the sphere
  $\mathbb{S}^{n-1}$ at infinity with its standard metric, and $h$ is independent of $x$. In general, if $(\partial M, h(0))$ is the sphere with its standard metric, then we
  call $(M^\circ,g)$ an asymptotically Euclidean metric. 

On a general asymptotically conic manifold, writing $r = 1/x$ one
obtains an analogue of the radial variable in this more general
context, and the metric then takes the form near infinity 
\begin{equation}\label{ac-metric2}
g = dr^2 + r^2 h(\frac1{r}, y, dy). 
\end{equation}

The admissible nonlinearities are those $N[u] \defeq N(u, \overline{u}, \nabla u, \nabla \overline{u}, \nabla^{(2)}
u, \nabla^{(2)} \overline{u})$ which are a sum of monomial terms, of degree not less than $p$,
 in $u$ and $\overline{u}$ and their derivatives up to order two, with
coefficients smooth on $M$. Moreover, we require $p$ to satisfy the first condition in \eqref{pkcond}. 
  

\begin{theorem}[Main Theorem, asymptotically conic case]\label{thm:main2}
Let $(M^\circ, g)$ be an asymptotically conic manifold of dimension
$n$, and let $V$ be a conormal short range potential, that is, a smooth potential on $M^\circ$ satisfying estimates near infinity of the form 
\begin{equation}
\Big| (r D_r)^j D_y^\alpha V(r, y) \Big| \leq C \langle r \rangle^{-\gamma} \text{ for all } j \geq 0, \ \alpha \in \mathbb{N}^{n-1}
\label{Vests}\end{equation}
for some $\gamma > 1$. 
Let $H = \Delta_g + V$ where $\Delta_g$ is the Laplace-Beltrami operator on
$(M^\circ, g)$. Let $N[u]$ be an admissible nonlinearity, and let $p$ and $k$ be
integers satisfying \eqref{pkcond}. 
 There exists $c
> 0$ sufficiently small, such that 
for every $f \in H^{k}(\partial M)$ with $\| f
\|_{H^{k}(\partial M)} < c$,  there is a solution $u$ to 
$$
(H - \lambda^2) u = N[u]
$$
on $M^\circ$ satisfying
\begin{equation}\label{eq:asymptotic real}
u(r, y) = r^{-(n-1)/2} \Big( e^{-i \lambda r}f(y) + e^{+i \lambda r}
b(y) + O_{H^{k -2}}(r^{-\epsilon}) \Big) 
\end{equation}
for some $b \in H^{k}(\partial
M)$ and some $\epsilon > 0$. 

Assume further that the stricter inequality $(p - 1)(n -1)/2 > 3$ holds and that the nonlinearity $N[u]$ involves derivatives
up to order one (instead of two as allowed above). Then we
again have the decomposition $b = S_{\mathrm{lin}}(\lambda) f + b_1$ with $S_{\mathrm{lin}}(\lambda) f$ again the
linear scattering matrix (now an FIO associated to geodesic flow for
time $\pi$ \cite{mezw96}) and $b_1 \in H^{k +
  1}(\partial M)$.  Again, if $f \in H^{k
  + j}(\partial M)$ (in addition to the smallness condition in $H^k$) then $b_1 \in H^{k + j + 1}(\partial M)$.

\end{theorem}

For ease of exposition, we return to the Euclidean case in the remainder of this introduction. 
Given $f \in L^2(\mathbb{S}^{n-1})$, the 
linear solution $u_0$ to $(H - \lambda^2) u_0 = 0$ with incoming data
$f$ is the image of $f$ under the \emph{incoming Poisson
operator}, $P(\lambda)$.  This $u_0$ can be written (non-uniquely) as a
sum of incoming and outgoing terms $u_0 = u_- + u_+$, where, roughly
speaking, $u_\pm \sim r^{-(n-1)/2} e^{\pm i r \lambda} f_\pm$ as $r$
goes to infinity.  For adequate decompositions $u_\pm$, solutions $u$ to \eqref{eq:equation}
satisfying \eqref{eq:asymptotic} can be constructed by a contraction
mapping argument in which one writes $u = u_- + w  = u_0 + (w - u_+)$
where $w$ will be outgoing (in a sense to be made precise below) 
and a fixed point of the mapping
$$
\Phi(w) = u_+ + R(\lambda + i0)N[u_- + w],
$$
where $R(\lambda + i 0)$ is the outgoing resolvent, that is, $(H - (\lambda + i0)^2)^{-1}$.  The true
nonlinear solution therefore satisfies
\begin{equation}
  \label{eq:true solution fixed}
  u = u_0 + R(\lambda + i0)N[u].
\end{equation}
At issue in this
paper is the regularity of the outgoing data of $w$, and how this can
be understood in terms of the mapping properties of the Poisson
operator, the decomposition $u_\pm$,  and the resolvent.

The construction of the nonlinear eigenfunction $u$ requires a
decomposition $u_0 = u_- + u_+$ for the free solution, and we improve
on the reults in \cite{ghsz20} by using the Schwartz kernel of the Poisson operator
$P(\lambda)$ (see Section \ref{sec:Pois}) to obtain a decomposition
with optimal regularity.  Indeed, the Poisson operator is given by the action of a
well understood oscillatory integral kernel, due to \cite{mezw96} in the asymptotically conic case, 
with many antecedents in the Euclidean case, e.g. \cite{Ik60, Is82}. 
Using this we prove a relationship between
the regularity of the incoming data $f \in H^k(\mathbb{S}^{n-1})$ for $k \ge
0$ and the
corresponding linear generalized eigenfunction $u_0$.  This is best understood using the theory of scattering
pseudodifferential operators  \cite{me94}.  It is well-known that
$u_0$ lies in weighted Sobolev spaces $H^{s, - 1/2 - \epsilon}(\mathbb{R}^n)$,
where $s \in \mathbb{R}$ and $\epsilon > 0$ are arbitrary.  This means
that
$$
\langle z \rangle^{- 1/2 - \epsilon} u_0 \in H^s(\mathbb{R}^n),
$$
where $H^s(\mathbb{R}^n)$ are the standard $L^2$-based Sobolev spaces on $\mathbb{R}^n$.  
(In particular, $u_0$ is a
smooth function in the interior, a consequence of elliptic regularity.)  What we prove below is that $u_0$ can be decomposed into $u_0 = u_-
+ u_+$ where each $u_\pm$ has $k$ additional order of `module
regularity', specifically each remains in $H^{s, -1/2-\epsilon}$
after application of $k-$fold combinations of 
angular derivatives and the radial annihilators $r(D_r\mp \lambda)$ of the oscillatory factors $e^{\pm i \lambda r}$.  As
we describe below, these operators are determined directly by the
microlocal structure of the problem; they comprise the modules
$\mathcal{M}_\pm$ of scattering pseudodifferential operators in
$\Psisc^{1,1}$ which are characteristic on the incoming/outgoing
($-$/$+$) \textit{radial} sets $\mathcal{R}_\pm$ of the operators $H - \lambda^2$, thought of as
a \emph{non-elliptic} scattering operator with non-degenerate
characteristic set over spatial infinity.

The paper is organized as follows.  In Section 2 we recall the basic
definitions and structures that will be used in the paper, including
scattering pseudodifferential operators and
the weighted Sobolev spaces between which they act.  We also recall there
the definitions of the module regularity spaces used in \cite{ghsz20}
and the crucial mapping properties of the resolvent between such
spaces.  In Section 3 we discuss the mapping properties of the Poisson
operator.  These properties are very closely related to mapping properties of the
incoming and outgoing resolvents, due to formula \eqref{sm}. 
The key result
is Proposition \ref{prop:reslimits} which shows that the resolvent applied to certain functions admit 
asymptotic expansions corresponding to one of the terms in \eqref{eq:asymptotic real}. The PDE-type
argument, however, does not give the optimal regularity for the leading coefficient. The optimal regularity is obtained 
in Propositions~\ref{prop:limit-reg} and \ref{prop:reg}, which relates the map taking a function $F$ to the leading expansion of
$R(\lambda \pm i0) F$ to the \emph{adjoint} of the Poisson operator, $P(\mp \lambda)^*$,  applied to $F$.   In
Section 4 we put these results together to prove the main theorem.

Our analysis of the Poisson operator is based on the description of
its Schwartz kernel as an oscillatory integral, which is due to Melrose-Zworski \cite{mezw96}.  The powerful
tools from microlocal scattering theory we employ were developed in 
\cite{me94} (radial point estimates), \cite{HMV04} (test modules), \cite{Va13} (Fredholm theory for real principal type operators on anisotropic Sobolev spaces) and 
\cite{va18} (radial point estimates and Fredholm theory in the scattering calculus), and we refer the reader to
\cite{ghsz20} for the detailed discussion of related literature.


\section{Microlocal analysis of the resolvents
  $R(\lambda \pm i 0)$}

We review the relevant objects here only briefly  as they are discussed
in detail in other work.  A detailed introduction to scattering
differential operators on $\mathbb{R}^n$ can be found in Vasy's
minicourse notes \cite{va18},
while the more general development on scattering manifolds is due to
Melrose \cite{me94}.  See also Sections 2 and 3 of \cite{ghsz20}.

\subsection{Weighted Sobolev spaces and scattering pseudodifferential
  operators}

We confine most of our introductory discussion to the case of
Euclidean space.  Letting $\mathcal{S}(\mathbb{R}^n)$ denote the space of Schwartz
functions and $\mathcal{S}'$ the tempered distributions,
each $u \in \mathcal{S}'$ lies in some weighted $L^2$-based Sobolev
space
\begin{equation}
  \label{eq:weighted spaces}
  H^{s,l}(\mathbb{R}^n) = \langle z \rangle^{- l} H^m(\mathbb{R}^n),
\end{equation}
for $s, l \in \mathbb{R}$.

Recall the scattering symbols and
scattering pseudodifferential operators, defined for $m, l \in
\mathbb{R}$ by
\begin{equation*}
  \begin{gathered}
 S^{s, l}(\mathbb{R}^n) = \{ a(z, \zeta) \in C^\infty(\mathbb{R}^n_z
 \times \mathbb{R}^n_\zeta) : \| a \|_{s, l ; N} < \infty  \mbox{ for
   all } N \in \mathbb{N}_0\} \\
 \Psisc^{s, l}(\mathbb{R}^n) = \{  \text{Op}(a)  : a \in S^{s, l}(\mathbb{R}^n)  \},
  \end{gathered}
\end{equation*}
where
$$
\| a \|_{s, l; N} = \sum_{|\alpha| +  |\beta| \le N} \sup_{z, \zeta} | \langle z
\rangle^{-l + |\alpha|}\langle \zeta \rangle^{-s + |\beta|} D_z^\alpha
D_\zeta^\beta a(z, \zeta)|,
$$
Here $\text{Op}(a)$ denotes the operator with integral kernel $\int e^{ - i (z
  - z')} a(z, \zeta) d\zeta$. Thus the
  scattering pseudodifferential operators are by definition (here) the
  left quantizations of scattering symbols.
  
  There is a notion of principal symbol attached to scattering
  PsiDO's which includes their leading order behavior at spatial infinity.  The (filtered) algebra of
  scattering PsiDO's admits a natural mapping to the graded algebra of
  scattering symbols, 
  $$
\Psisc^{s, l} \lra S^{[s, l]} := S^{s, l} / S^{s - 1, l - 1}.
$$
Given $A \in \Psisc^{s, l}$ its principal symbol will be denoted by
$\sigma_{s, l}(A)$.

We work here exclusively with classical scattering symbols, which are
functions $a(z, \zeta)$ with joint asymptotic expansions in $z, \zeta$
as $|z|, |\zeta| \to \infty$.  We will need to construct explicit
symbols whose corresponding operators will function as microlocal
cutoffs.  In general, classical symbols can be written
  $$
a \equiv \langle z
\rangle^{l}\langle \zeta \rangle^{s} a_0(z, \zeta) ,
  $$
where $a_0(z, \zeta)$ is bounded and is smooth in $r^{-1} =
|z|^{-1}$ and $|\zeta|^{-1}$.  This
characterization of the asymptotic behavior of $a_0$ is
equivalent to the statement that $a_0$ extends to an element in
$C^\infty(\overline{\mathbb{R}^n} \times \overline{\mathbb{R}^n})$
where $\overline{\mathbb{R}^n}$ is the radial compactification of
$\mathbb{R}^n$ in which the map $z \mapsto z / (1 + \langle z
\rangle)$ realizes $\mathbb{R}^n$ as the interior of the unit ball
$\mathbb{B}^n$, and the overline notation denotes this entire structure (the
map from Euclidean space into the ball as opposed to just the ball.)

For $[a] \in S^{[s, l]}$, i.e.\ $[a]$ the principal symbol of a
scattering PsiDO, the value of $a_0$ is determined here only to
the addition of elements in $S^{s-1, l -1}$, so in particular, in sets of bounded frequency,
$|\zeta| < C$,  with $\supp (\chi \colon \mathbb{R}_{t} \to \mathbb{R})
\subset \{ |t| \le C \}$,
$$
\chi(\zeta) (a_0(z, \zeta) - a(\hat z , \zeta))  \equiv 0 \qquad  (\mbox{mod } S^{s
  -1, l - 1}),
$$
with a similar expression in the $|z| < C$ regions.  
An $A \in \Psisc^{s, l}$ is by definition \emph{scattering elliptic} if its
principal symbol is invertible in the graded algebra $\bigcup_{s, l} S^{[s,l]}$.  The microlocal notion
of scattering ellipticity will be used as follows.  On regions of
large frequency $|\zeta| > C$, scattering ellipticity is implied by
the uniform estimate
$$
\sigma_{s, l}(A)(z, \zeta) \ge C \langle z \rangle^l \langle \zeta
\rangle^s \mbox{ for } |\zeta| > C.
$$
Below, the operators of interest are (scattering) elliptic on such large frequency
regions but are not globally scattering elliptic, hence the more
detailed estimates coming from propagation
phenomena will arise from analysis on sets of bounded frequency, whence we will
typically need only to discuss the ``spatial'' principal symbol of $A
\in \Psisc^{m,l}$, defined for classical scattering symbols by
$$
\sigma_{\text{base},\,  l}(A)(\hat z, \zeta) = \lim_{r \to \infty} \langle z
\rangle^{-l} a(z, \zeta).
$$
Note that the behavior in $\zeta$ of this function is in general symbolic but not
homogeneous.  In this region, (scattering) ellipticity is the same familiar
ellipticity but in the $z$ directions; $(\hat z_0, \zeta_0)$ is in the
elliptic set if $\sigma_{s,l}(A)(z, \zeta) \ge C \langle z \rangle^l$
in a region $|z| > C, |\hat z - \hat z_0|, |\zeta - \zeta_0| <
\epsilon$.  It is straight forward to check that this is equivalent to
$\sigma_{\text{base},\,  l}(A)(\hat z_0, \zeta_0) \neq 0$.

The basic boundedness property of scattering
pseudodifferential operators is that for $A \in \Psisc^{s', l'}$,
\begin{equation}
  \label{eq:bounded psidos}
  A \colon H^{s, l} \lra H^{s -s', l - l'} \mbox{ is bounded.}
\end{equation}
The residual operators are
$$
\Psi^{- \infty, - \infty} = \bigcap_{m , l} \Psisc^{m, l},
$$
which have Schwartz kernels in $\mathcal{S}(\mathbb{R}^n \times
\mathbb{R}^n)$, and in particular $A \in \Psisc^{- \infty, -\infty}$ is bounded
between \emph{any} two weighted $L^2$-based Sobolev spaces; in
particular since $H^{s, l} \subset H^{s', l'}$ for any $s > s', l >
l'$, residual $A$'s define compact operators on $L^2$.
If $A$ is globally scattering elliptic then the map in \eqref{eq:bounded psidos} is
Fredholm, since in that case there is an approximate inverse
$B \in \Psisc^{-s, -l}$ such that $\Id - AB, \Id - BA \in \Psisc^{-
  \infty, -\infty}$.

All of the foregoing material extends directly to the general
asymptotically conical case; see \cite[Section 2.2]{ghsz20} for
further details.  In the general case, one has the fiberwise radial
compactification of the scattering cotangent bundle ${}^{\sca}
\overline{T}^* M$ which is a manifold with corners of codimension
  $2$; there are boundary defining fuctions $x$ for spatial infinity
  and $\rho$ for fiber infinity (which in the Euclidean case can be
  taken to be $\langle z \rangle^{-1}$ and $\langle \zeta
  \rangle^{-1}$, respectively) and the symbol estimates are
  essentially the same as the Euclidean space symbol estimates written in
  terms of $x$ and $\rho$.  The scattering pseudodifferential
  operators $\Psisc^{s,l}(M)$ are the quantizations of these symbols,
  and the scattering Sobolev spaces, denoted $H^{s, l}_{\sca}(M)$,
  consist of distributions $u$ with $A u \in
L^2(M)$ for all $A \in \Psisc^{s,l}(M)$.

\subsection{Mapping properties of $H - \lambda^2$ and the resolvent}\label{sec:resolvent}
Analysis of $H - \lambda^2$ as a scattering differential operator was
first carried out by Melrose in \cite{me94}. We review the relevant
material again on $\mathbb{R}^n$.  We have
$H  - \lambda^2 \in \Psisc^{2,0}$, with
$$
\sigma_{2,0}(H  - \lambda^2) = |\zeta|^2 + V -\lambda^2 \ge C\langle \zeta \rangle^2, \quad \text{for $|\zeta| > C$}.
$$
In the general asymptotically Euclidean case, we get that the spatial principal symbol is
rather simple:
\begin{equation}
  \label{eq:scat fib symb helm}
  \sigma_{\text{base},\,  2}(H- \lambda^2)(\hat z, \zeta) = |\zeta|^2 - \lambda^2. 
\end{equation}
\textit{In particular $H - \lambda^2$ is not globally scattering elliptic.}  Its
characteristic set $\Sigma$, which is a subset of $\mathbb{S}^{n -
  1}_{\hat z} \times \mathbb{R}^n_\zeta$, is the vanishing locus of
the fiber principal symbol:
\begin{equation}
  \label{eq:dumb char set}
  \Sigma = \{(\hat z, \zeta) :  |\zeta|^2 - \lambda^2  = 0 \}.
\end{equation}

To understand the action of the flow of the Hamilton vector field on
$\Sigma$, we can work in polar coordinates $(r, y)$, where
$y$ are arbitrary coordinates in a coordinate patch on the sphere $\mathbb{S}^{n - 1}$;
these automatically give rise to dual coordinates $(\nu, \eta)$, and
we rescale the angular dual variable $\eta$ by writing $r \mu
= \eta$.  Via this localization and rescaling we effectively pass to the case of
an arbitrary asymptotically conic manifold.

Noting that in these variables we have 
$$
|\zeta|^2 = \nu^2 + |\mu |_h^2,
$$
 where $| \mu |_h$ denotes the norm with respect to the dual of metric $h$ in the definition of the asymptotically conic metric $g$. Apart from giving the obvious rewriting of the
characteristic set as $\Sigma = \{ \nu^2 + |\mu |_h^2 =
\lambda^2 \}$, it clarifies the behavior of the Hamilton vector field,
which, recall, is defined with respect to arbitrary coordinates and
their canonical dual coordinates to be
$$
H_p := \frac{\partial p}{\partial \nu} \frac{\partial}{\partial r}
- \frac{\partial p}{\partial r} \frac{\partial}{\partial \nu} +
\sum_{j = 1}^{n-1} \Big( \frac{\partial p}{\partial \tilde \mu_j} \frac{\partial}{\partial y_j}
- \frac{\partial p}{\partial y_j} \frac{\partial}{\partial \tilde
  \mu_j} \Big).
$$
Scattering calculus propagation results are phrased in terms of the
natural conformal rescaling (or reweighting) of this vector field, namely $\mathsf{H}_p  := \langle z \rangle H_p,$
and with $p = \sigma_{2, 0}(H - \lambda^2)$, we obtain 
\begin{equation}\label{flat-Hvf}
\mathsf{H}_p = -2\nu (x \partial_x + R_\mu) + 2|\mu|_h^2 \partial_\nu
+ H_{\partial M, h},
\end{equation}
where $R_\mu$ is the radial vector field in $\mu$ and $H_{\partial
  M, h}$ is the Hamilton vector field on $\partial M$ for the metric
$h$.  (Note that the apparent discrepancy in sign between
\eqref{flat-Hvf} and \cite[Eq.\ 8.19]{me94} is due to our usage of
$\nu$, the dual variable to $r$, in contrast with Melrose's use of
$\tau = - \nu$.)

From this one deduces that the two submanifolds
\begin{equation}
  \label{eq:radial sets}
  \mathcal{R}_\pm := \{ \mu = 0 = x, \ \nu = \pm \lambda \}
\end{equation}
are sinks ($+$) or sources ($-$) for the (rescaled)
Hamilton flow, i.e.\ the flow of $\mathsf{H}_p$, on the characteristic
set  $\Sigma$.  The well-known formula for the principal symbol of a
commutator of pseudodifferential operators also extends to the
scattering setting, namely if $A \in \Psisc^{s, l}, B \in \Psisc^{s',
  l'}$ with $a = \sigma_{s, l}(A), b = \sigma_{s', l'}(B)$, then $[A, B] \in \Psisc^{s + s' - 1, l + l' - 1}$ and
$$
\sigma_{s + s' - 1, l + l' - 1}(i [A, B]) = \{ a, b \} = H_a b,
$$
where here $\{ \cdot, \cdot \}$ denote the Poisson bracket.

The most important example of the use of commutators for us will be
those of the form $[H ,Q]$ where $H$ is the Hamiltonian and $Q \in
\Psisc^{0,0}$, in particular we will take $Q$ to have various
microsupport properties (discussed below) which will allow them to act
as microlocal cutoffs to various portions of scattering phase space.
Particularly
simple examples include symbols of the form
$$
q(x, y, \nu, \mu) = \chi(x) \chi((|\mu|_h^2 + \nu^2)/ 2 \lambda^2)
\phi(\nu), \qquad \phi \in C^\infty_c(\R^n).
$$
where $\chi$ is a smooth 
function with $\chi(t) =1$ for $t \le 1$ and
$\chi(t) = 0$ for $t \ge 2$.  (This $q$ is then a function whose
support lies in a neighborhood of $\Sigma$ and whose restriction to $\Sigma$
depends only on $\nu$.)  Then in fact $Q \in \Psisc^{-\infty, 0}$ and
$[H, Q] \in \Psisc^{-\infty, -1}$, whence
\begin{equation}
  \label{eq:3}
  \begin{gathered}
\langle z \rangle  \sigma_{- \infty, -1}(i [H, Q]) = \mathsf{H}_p q =
2|\mu|_h^2 \phi'(\nu) \mbox{ in the region } |x| < 1, \ |\mu|_h^2 +
\nu^2 < 2 \lambda^2,\\
\mbox{and } \langle z \rangle \sigma_{- \infty, -1}(i [H, Q]) = 0 \mbox{ for } |\mu|_h^2 + \nu^2 \ge 2\lambda.
\end{gathered}
\end{equation}

We can now define the modules of pseudodifferential operators
$\mathcal{M}_\pm$ used to measure the regularity of distributions in
the domain and range of our formulation of the resolvent mapping.
Specifically
\begin{equation}
  \label{eq:2}
  \mathcal{M}_\pm := \{ A \in \Psisc^{1,1} : \mathcal{R}_\pm \subset \Sigma_{1,1}(A) \},
\end{equation}
or, in words, $\mathcal{M}_+$ is the vector space of scattering
pseudodifferential operators of order $(1,1)$ which are characteristic
on the ``outgoing'' radial set $\mathcal{R}_+$ and $\mathcal{M}_-$ is
the same for $\mathcal{R}_-$.  Locally an element $A \in
\Psisc^{1,1}$ lies in $\mathcal{M}_\pm$ if and only if it can be written
$$
A = B_0 r(D_r \mp \lambda) + \sum_{i = 1}^{n-1} B_i  D_{y_i} + B'
$$
with $B_0, B_i, B' \in \Psisc^{0,0}$.  Note that $\mathcal{M}_+$ and $\mathcal{M}_-$
both contain the identity operator. Also, $\mathcal{M}_\pm$ contain a elliptic element of
order $(1,1)$, namely the radial operator $r(D_r \mp \lambda)$, at the opposite radial set
$\mathcal{R}_\mp$. 

We also need the `small' module of angular derivatives
$$
\mathcal{N} = \{ \sum_{i = 1}^{n-1} B_i  D_{y_i} + B' : B_i, B' \in \Psisc^{0,0} \} = \mathcal{M}_+ \cap \mathcal{M}_-. 
$$

Theorem 2.7 of \cite{ghsz20} then gives the following mapping property for
the resolvent.  Let $s, l \in \mathbb{R}, \kappa, k \in \mathbb{N}_0$, and define
\begin{equation}
  \label{eq:mod reg space defs}
  \begin{gathered}
    \mathcal{Y}^{s, l; \kappa, k}_\pm := H^{s, l ; \kappa, k}_\pm = \{ u \in H^{s, l}
    : \mathcal{N}^{k}\mathcal{M}_\pm^\kappa u \subset H^{s, l} \} \\
    \X^{s, l; \kappa, k}_\pm := \{ u \in  H_\pm^{s, l; \kappa, k} : (H -
    \lambda^2) u \in H_\pm^{s - 2, l + 1 ; \kappa, k} \}.
  \end{gathered}
\end{equation}
Then
\begin{equation}
  \label{eq:res map mod reg}
 R(\lambda \pm i0) \colon \mathcal{Y}_\pm^{s, l + 1;  \kappa, k} \lra
  \X_\pm^{s + 2 , l ; \kappa, k} \mbox{ provided } k \ge 1, \ l \in (-\frac1{2} -
  k, -\frac1{2}).
\end{equation}
Here, the value $l = -1/2$ is referred to as the ``threshold'' value; it is the critical spatial
order for which we need different ``radial propagation estimates'' for $l$ greater than vs. 
less than $-1/2$. See \cite[Section 3]{ghsz20}. The condition in \eqref{eq:res map mod reg}
is that $l$ is below threshold, but that $l + k$ is above threshold. That means that module
regularity for $\mathcal{M}_\pm$ of order $k$ gives above threshold regularity at the opposite
radial set $\mathcal{R}_\mp$, which is the key to obtaining estimates such as \eqref{eq:res map mod reg}.
This statement can be rephrased as follows: for such $k \in \mathbb{N}$ and $l \in \mathbb{R}$,
the operator $H - \lambda^2$ is an isomorphism from 
  $\X_\pm^{s , l ; k, k'}$ to $\mathcal{Y}_\pm^{s - 2, l + 1;
    k, k'}$, and its inverse is the resolvent $R(\lambda \pm i0)$.

The reason for introducing the small module $\mathcal{N}$ to treat
the nonlinear problem is explained in detail in the introduction of
\cite{ghsz20}.  We remark here only that for multiplication of
distributions in $H_+^{s, l; \kappa, k}$, large module regularity produces loss in decay -- see
e.g. \cite[Cor.\ 2.10]{ghsz20} from which one has $H_+^{s, l; \kappa, k}
\cdot H_+^{s, l; \kappa, k} \subset H_+^{s, 2l + (n/2)  - \kappa; \kappa, k}$.  Small
module regularity is used to minimize this loss of $\kappa$ in the spatial order.

  \subsection{Microlocalization}\label{sec:micro}

  It will be important to microlocalize distributions both near to and
away from the radial set.  To this end we recall some features of the operator wavefront set $\WF'(A)$ of $A \in
\Psisc^{s,l}(M)$.  We will work mostly with those $A$ which are
compactly microlocalized in frequency.  Concretely, this means that
$$
A =  \text{Op}(a(z, \zeta)) + E, \quad \supp a(z, \zeta) \subset \{
|\zeta| < C \}, \ \kappa_E \in \mathcal{S}(\mathbb{R}^n \times \mathbb{R}^n),
$$
for some $C > 0$ ($\kappa_E$ is the Schwartz kernel of $E$).
The condition on $\kappa_E$ is equivalent to $E$ being a \emph{residual} operator, meaning
$$
E \in \Psisc^{-\infty, -\infty}  := \bigcap_{m, l \in \mathbb{R}
  \times \mathbb{R}} \Psisc^{m, l}.
$$
For such $A$, $\WF'(A) \subset \mathbb{S}^{n-1}_{\hat
  z} \times \mathbb{R}^n_{\zeta}$ is by definition the complement of
the set $(\hat z, \zeta)$ such that $a$ is Schwartz in $\zeta$ in the
$\hat z$ direction. 

Clearly $\WF'(A) = \varnothing \implies A \in \Psisc^{-\infty, -
  \infty}$, and thus for any $S, L \in \mathbb{R}$,
$$
\WF'(A) = \varnothing \implies A \colon H^{S, L}(\mathbb{R}^n) \lra
\mathcal{S}(\mathbb{R}^n).
$$
Also, wavefront sets have the expected algebraic property of supports,
namely, for $A \in \Psisc^{s,l}, B \in \Psisc^{s', l'}$, 
$$
\WF'(A B) \subset \WF'(A)\cap \WF'(B).
$$

Two particularly useful examples of such operators are those of the
form $Q_\pm \in \Psisc^{0,0}$ which are microlocalized near the two
components of the radial set $\mathcal{R}_\pm$.  This can be done
by defining $Q_\pm$ explicitly in a way similar to the definition of
$Q$ above \eqref{eq:3}, specifically one can take
\begin{equation}
  \label{eq:q pm}
  Q_\pm = \text{Op}(q_\pm), \quad q_\pm(x, y, \nu, \mu) =
\chi(x) \chi((|\mu|_h^2 + \nu^2)/ 2 \lambda^2) \chi((\nu \mp \lambda)/\epsilon),
\end{equation}
where again $\chi(s) =1$ for $s \le 1$ and
$\chi(s) = 0$ for $s \ge 2$.

Alternatively, one can simply break up phase space into a region where
$\nu$ is positive and its complement, excluding one radial set from each.  
We choose an operator $A_+ \in \Psi^{0,0}_{\text{cl}}(M)$ such that $A_+$ is
microlocally equal to the identity in a neighbourhood of $\Rp$, and
microlocally equal to $0$ in a neighbourhood of $\Rm$. We also let
$A_- = \Id - A_+$; thus, $A_-$ is microlocally equal to the identity
in a neighbourhood of $\Rm$, and microlocally equal to $0$ in a
neighbourhood of $\Rp$. It is convenient to choose $A_+$ such that its
principal symbol $a$ is a function only of $\nu$ in a neighbourhood of
the characteristic variety of $H$, and is monotone. Indeed, we can
take 
\begin{equation}
A_+ = \text{Op}(a), \quad a(x, \nu) = \chi(x) \chi((|\mu|_h^2 + \nu^2)/ 2 \lambda^2) \tilde \chi(\nu),
\label{eq:main commutator}\end{equation}
where $\chi$ is as in the definition of $Q_\pm$ and $\tilde \chi(\nu)
\equiv 1$ in $\nu > \lambda / 4$ and $\tilde \chi(\nu) \equiv 0$ for
$\nu < -\lambda/4$, and is monotone in between.

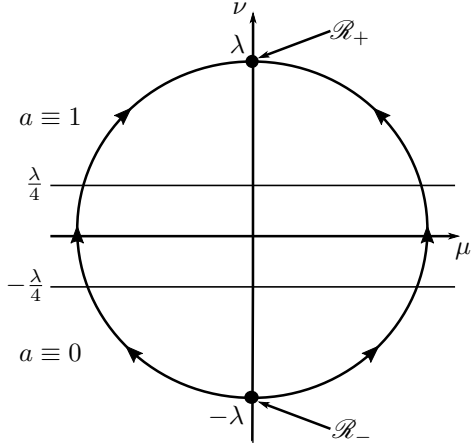
\begin{figure}
\centering
\labellist
\small
\pinlabel $\mathcal{R}_-$  at  134 5
\pinlabel $\mathcal{R}_+$  at  133 182
\pinlabel $\frac{\lambda}{4}$ at -7 115
\pinlabel $-\frac{\lambda}{4}$ at -11 70
\pinlabel $a\equiv1$ at 0 145
\pinlabel $a\equiv0$ at 0 40
\pinlabel $\lambda$ at 83 177
\pinlabel $-\lambda$ at 78 12
\pinlabel $\mu$ at 183  85
\pinlabel $\nu$ at 84 193
\endlabellist
\includegraphics[scale=.85]{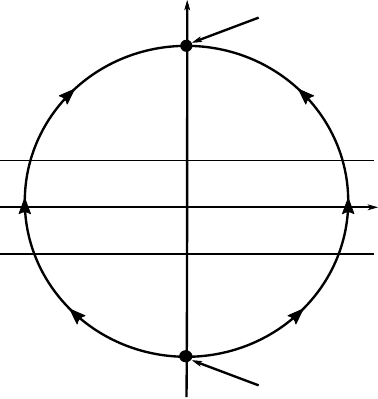}
 \caption{The bicharacteristic flow in the characteristic set \newline
   $\Sigma(P) = \{x = 0, \, \nu^2 + |\mu|_h^2 = \lambda^2\}$ with
   $P=\Delta-\lambda^2$ and $a$ the principal symbol of the microlocalizer
 $A_+$.}
 \label{fig: char set}
\end{figure}

While microlocalization gives us a concrete mechanism for analyzing
frequency-localized spatial decay, it will be useful to note that
this is also possible using anisotropic Sobolev spaces, meaning
Sobolev spaces $H^{s, \vl}$ in which the parameters $s, \vl$ are
themselves functions on phase space.  Since the operators under
consideration here are elliptic outside a compact set in frequency,
we will always take $s$ constant, but we shall employ variable spatial weights
$\vl_\pm \in S^{0,0}$ satisfying
  \begin{equation}\label{eq:decay function}
\begin{gathered}
\text{$\vl_+$ is equal to $-1/2$ outside small neighbourhoods of $\Rp$ and $\Rm$,}   \\  
\text{$\vl_+$ equals $-1/2 \mp \delta$ in smaller neighbourhoods of $\mathcal{R}_{\pm}$, }\\
 \text{$\vl_+$ is nonincreasing along the Hamilton flow of $P$ within $\Sigma(P)$,} \\
 \text{and $\vl_- = -1 - \vl_+$ has corresponding properties with $\Rm$ and $\Rp$ switched.}
\end{gathered}
\end{equation}
For definiteness, we suppose that for some sufficiently small $\delta \in [0, 0.01]$, we have $\vl_+$ equal to $-1/2$ outside the sets $\{ x^2 + |\mu|_h^2 + (\nu \mp \lambda)^2 \leq 4\delta \lambda^2 \}$,
and is equal to $-1/2 \mp \delta$ within the sets $\{ x^2 + |\mu|_h^2 + (\nu \mp \lambda)^2 \leq \delta \lambda^2 \}$. 
Thus, distributions in $H^{s, \vl_+}$ lie in $H^{s, -1/2
  - \epsilon}$ and are microlocally $H^{s,
  -1/2 + \epsilon}$ near $\mathcal{R}_-$; that is, they are below threshold regularity at $\Rp$ and above threshold at $\Rm$, with a corresponding statement true with reversed signs for $H^{s, \vl_-}$. 
Then by Theorem 3.2 of \cite{ghsz20}  we have isomorphisms
\begin{equation}\label{eq:var spaces}
  \begin{gathered}
    H - \lambda^2 \colon \X^{s, \vl_\pm} \lra \Y^{s - 2,
      \vl_\pm + 1}, \\
    \Y^{s - 2,      \vl_\pm + 1} = H^{s - 2,    \vl_\pm + 1},
    \quad \X^{s, \vl_\pm} = \{ u \in H^{s, \vl_\pm} : (H -
    \lambda^2) u \in \Y^{s - 2,      \vl_\pm + 1}\}.
  \end{gathered}
\end{equation}
The inverse maps are the incoming/outgoing resolvent $R(\lambda \pm i0)$. 
As with \eqref{eq:res map mod reg}, we notice 
the sharp difference of one in the spatial regularity, and the avoidance of the threshold value at the radial sets. 

Moreover, we can combine variable order spaces with module regularity, obtaining isomorphisms 
\begin{equation}\label{eq:var spaces modreg}
  \begin{gathered}
    H - \lambda^2 \colon \X_\pm^{s, \vl_\pm; \kappa, k} \lra \Y_\pm^{s - 2,
      \vl_\pm + 1; \kappa, k}, \\
    \Y_\pm^{s', \vl'; \kappa, k} :=  \{ u \in \Y_\pm^{s', \vl'} \mid 
    \mathcal{N}^{k}\mathcal{M}_\pm^\kappa \, u \subset \Y_\pm^{s', \vl'} \} \\
    \X_\pm^{s, \vl_\pm; \kappa, k} :=  \{ u \in  \Y_\pm^{s, \vl_\pm; \kappa, k} : (H -
    \lambda^2) u \in \Y_\pm^{s - 2, \vl_\pm + 1 ; \kappa, k} \}.
   \end{gathered}
\end{equation}
We emphasize that in \eqref{eq:var spaces modreg}, we can allow $\kappa=0$, that  is, only consider small module
regularity, in contrast to \eqref{eq:res map mod reg}. 

These variable order module regularity results do not appear explicitly in our previous work \cite{ghsz20}, but such
results follow readily from the propagation estimates in Section 3 of that paper. (The reason for
restricting the module regularity spaces to constant orders in \cite{ghsz20} is that it is more convenient to analyze multiplicative
properties of such spaces when the orders are constant.)
We shall not review variable order spaces in detail here, referring the
reader to \cite{va18}. Here we only mention one elementary property of such spaces. 
Namely, we can conclude
containment in these types of variable order spaces using
microlocalization. For example, it is straightforward to show that, given
a tempered distribution $u$, for $\epsilon > 0$ sufficiently small and $Q_+ \in \Psisc^{0,0}$ as in \eqref{eq:q pm},
  if $Q_+ u \in  H^{s, -1/2 - \delta} $ and $(\Id - Q_+) u \in
  H^{s, -1/2 + \delta}$, then $u \in H^{s, \vl_+}$.

  We note that in all these considerations, the value of $\vl_\pm$ off
  $\Sigma$ is essentially irrelevant, since off $\Sigma$ one has
  elliptic estimates for $H - \lambda^2$.


\section{Mapping properties of the linear Poisson operator}\label{sec:Pois}

The Poisson operator furnishes a distorted or generalized Fourier transform for the operator $H$. Constructions
in the Euclidean case go back a long way in scattering theory, see e.g. \cite{Ik60}, \cite{Is82}. We use the results of 
Melrose and Zworski \cite{mezw96}, who constructed the Poisson
operator $P(\lambda)$, $\lambda > 0$, for the Helmholtz operator $H =
\Delta - \lambda^2$ on asymptotically conic manifolds (scattering metrics). 
Their results extend readily to the case of a Schr\"odinger operator with potential in the
conormal space $\mathcal{A}^\gamma(M)$, for $\gamma > 1$, which is by definition the space of 
smooth functions $V$ on $M^\circ$ satisfying estimates near $\partial M$ of the form 
\begin{equation}
\Big| (x D_x)^j D_y^\alpha V(x, y)  \Big| \leq C x^{\gamma}, \quad j \geq 0, \ \alpha \in \NN^{n-1}. 
\label{Vconormal}\end{equation}
These estimates on $V$ are equivalent to \eqref{Vests} in the case of Euclidean space. 

The Poisson operator has the property that it maps a smooth function $\phi \in C^\infty(\partial M)$ to a function $u_0$ satisfying $Hu_0 = 0$, with the asymptotics
\begin{equation}
u_0 = u_- + u_+, \quad 
u_\pm = r^{-(n-1)/2} e^{\pm i\lambda r} f_\pm, \quad f_\pm \in C(M), \quad f_- \big|_{\partial M} = \phi. 
\label{usmooth}\end{equation}
In fact the $f_\pm$ can be taken to lie in  $C^\infty(M) + \mathcal{A}^{\gamma - 1}(M)$ (which are continuous up to the boundary, as $\gamma -1 > 0$); moreover, if $V$ is equal to $r^{-2}$ times a smooth function on $M$, then the $f_\pm$ can be taken in $C^\infty(M)$. 
We shall also define $P(-\lambda)$, again for $\lambda > 0$, to be the operator mapping $\phi \in C^\infty(\partial M)$ to a function $u_0$ satisfying $Hu_0 = 0$, with the asymptotics
\begin{equation*}
u_0 = u_- + u_+, \quad 
u_\pm = r^{-(n-1)/2} e^{\pm i\lambda r} f_\pm, \quad f_\pm \in C^\infty(M) + \mathcal{A}^{\gamma - 1}(M), \quad f_+ \big|_{\partial M} = \phi. 
\end{equation*}
Recall that $C^\infty$ functions on $M$, as opposed to $C^\infty$ functions on $M^\circ$, are (by definition) smooth functions of $x = 1/r$ and $y$ near $\partial M$.

The adjoint operator $P(\lambda)^*$ will also play a role in this article. To describe it, let $v$ be a Schwartz function\footnote{By a Schwartz function on $M$ we mean a smooth function that vanishes, with all its derivatives, at the boundary. When $M$ is the radial compactification of $\RR^n$ this corresponds to the usual meaning of Schwartz function.} on $M$, and let $u_\pm =  \pm R(\lambda \pm i0) v$. Then each $u_\pm$ has a similar expansion as in \eqref{usmooth}, but with only the outgoing ($+$)/incoming($-$) oscillation \cite[Proposition 12]{me94}:
\begin{equation}
u_\pm = r^{-(n-1)/2} e^{\pm i\lambda r} f_\pm, \quad f_\pm \in C^\infty(M) + \mathcal{A}^{\gamma - 1}(M).
\end{equation}
We claim that $P(\lambda)^* v = 2i\lambda f_- |_{\partial M}$. To see that this is true, we use a pairing identity that follows from Green's formula. Suppose that $u_1$ and $u_2$ are two functions of the form 
$$
u_i = r^{-(n-1)/2} \Big( e^{- i\lambda r} f_{i,-} + e^{+i\lambda r} f_{i,+} \Big), \quad  f_{i, \pm} \in C^\infty(M) + \mathcal{A}^{\gamma - 1}(M). 
$$
Suppose further that $(H - \lambda^2) u_i = O(r^{-(n+1)/2 - \epsilon})$ for some $\epsilon > 0$. Then the following identity holds \cite[Equation (13.1)]{me94} (see also \cite[Section 5]{HV99}):
\begin{equation*}
\int_M \Big(u_1 \overline{(H - \lambda^2) u_2} - ((H - \lambda^2) u_1) \overline{u_2} \Big) d\vol_g = 2i\lambda \int_{\partial M} \Big( f_{1,+} \overline{f_{2,+}} -  f_{1,-}\overline{f_{2,-}} \Big) d\vol_h ,
\end{equation*}
where we take the restrictions to $\partial M$ in the integral over the boundary.
To prove this one notes that the LHS is absolutely integrable, so it can be obtained as the limit, as $R \to \infty$, of the integral restricted to $\{ r \leq R \}$. Then one applies Green's
formula and uses the asymptotic form of the functions $u_i$ to see that the limit as $R \to \infty$ is the RHS.

We apply this with $u_1 = P(\lambda) a$, $a \in C^\infty(\partial M)$ and $u_2 = u_-$ above. Then $(H - \lambda^2)u_1 = 0$, and $f_{2,+} = 0$, while $f_{1,-} = a$. We obtain 
\begin{equation}
 \int_M (P(\lambda) a)  \overline{v} = 2i\lambda \int_{\partial M}  a  \overline{f_-}.
\end{equation}
This immediately yields 
\begin{equation} \label{poisson adjoint plus} 
P(\lambda)^* v = 2i\lambda f_- |_{\partial M}. 
\end{equation}
A similar argument with $u_+$ shows
\begin{equation} \label{poisson adjoint minus}
P(-\lambda)^* v = 2i\lambda f_+ |_{\partial M}.
\end{equation}
Moreover, applying $P(\pm \lambda)$ to $f_\mp$, we obtain $u_+ + u_-$, since this is the unique eigenfunction with incoming/outgoing data equal to $f_\mp$. It follows that we have 
\begin{equation}
P(\pm \lambda) P(\pm \lambda)^* = 2\lambda i  \Big( R(\lambda + i0) - R(\lambda - i0) \Big),
\label{sm}\end{equation}
acting on Schwartz functions. Notice that the quantity in \eqref{sm} is equal to \
the spectral measure, up to a factor of $d\lambda/4\pi$. Note also the simple consequence of \eqref{poisson adjoint plus} and \eqref{poisson adjoint minus}:
\begin{equation}
\text{$P(\pm \lambda)^*$ maps Schwartz functions on $M$ to smooth functions on $\partial M$. }
\label{eq:Padsmooth}\end{equation}

Our first purpose in this section to analyze the range of $P(\pm \lambda)$ on the $L^2$-based Sobolev space $H^m(\partial M)$. 
We begin with a basic mapping property for the Poisson operator on $L^2(\partial M)$.  For
any $\delta > 0$, identity \eqref{sm} implies that we have a continuous mapping
\begin{equation}
P(\pm \lambda) \colon L^2(\partial M) \lra H^{s, -1/2 -
  \delta}(M)\label{eq:poisson con't}
\end{equation}
Indeed, first note that both $R(\lambda \pm i0)$ map $H^{-1, 1/2  + \delta}(M)$
to $H^{1, -1/2 - \delta}(M)$.  This result is well known but follows
in particular from \eqref{eq:var spaces} taking $s = 1$ since for either choice of
$\pm$ and any $s$,
$$
H^{s-2, 1/2  + \delta} \subset \Y_\pm^{s - 2,
      \vl_\pm + 1}, \qquad  \X_\pm^{s, \vl_\pm} \subset H^{s, -1/2 - \delta}.
    $$
Formula \eqref{sm} then shows that $P(\pm \lambda)$ extends from $C^\infty(\partial M)$ to a bounded      
map from $L^2(\partial M)$ to $H^{1, -1/2 - \delta}(M)$, as this holds if and only if $P(\pm \lambda)
P(\pm \lambda)^*$ maps $H^{-1, 1/2  + \delta}(M)$ to $H^{1, -1/2 -
  \delta}(M)$ (since $H^{-1, 1/2  + \delta}(M)$ is the dual of  $H^{1,
  -1/2 - \delta}(M)$), whence \eqref{eq:poisson con't} holds for $s = 1$. Finally, since the image of
$P(\lambda)$ is contained in solutions to $(H - \lambda^2) u = 0$, the
differential index $1$ can be replaced by any $s$ using elliptic regularity.

Now we note that \eqref{eq:poisson con't} can be improved easily using
the microlocalization discussed above.  If we define
\begin{equation}
\vl_{\min} = \min(\vl_+, \vl_-), \quad \vl_{\max} = \max(\vl_+,
\vl_-),\label{eq:l min l max}
\end{equation}
then these are smooth functions on phase space for $\delta' > 0$
sufficiently small, and (again for arbitrary $s$),
$$
(H^{s, \vl_{\min}})^* = H^{-s, -\vl_{\min}} = H^{-s, \vl_{\max} + 1}
$$
by the final property in \eqref{eq:decay function}.  Repeating our
argument to deduce mapping properties for $P(\lambda)$ from
\eqref{sm}, we use
$$
H^{s-2, \vl_{\max} + 1} \subset \Y_\pm^{s - 2,
      \vl_\pm + 1}, \qquad \X_\pm^{s, \vl_\pm} \subset H^{s, \vl_{\min}}
    $$
to conclude that for any $s \in \mathbb{R}$,
\begin{equation}
  \begin{gathered}
    P(\pm \lambda) \colon L^2(\partial M) \lra H^{s, \vl_{\min}}(M), \\
    P(\pm \lambda)^* \colon H^{s, \vl_{\max} + 1}(M) \lra L^2(\partial M),
  \end{gathered}
\label{eq:poisson con't better}
\end{equation}
the latter following from adjunction of the former.  This improves \eqref{eq:poisson con't} since $\vl_{\text{min}} = -1/2 > -1/2
- \delta$ away from the radial sets.

Again we can phrase this in terms of microlocalization, giving improved decay
for $P(\pm \lambda) f$, $f \in L^2(\partial M)$ away from the radial
sets.  That is to say, if $Q \in \Psisc^{0,0}$ has $\WF'(Q) \cap
\mathcal{R}_+ \cup \mathcal{R}_- = \varnothing$, then
\begin{equation}
Q P(\pm \lambda) \colon L^2(\partial M) \lra H^{s,
  -1/2}(M),\label{eq:first improvement}
\end{equation}
as follows immediately from \eqref{eq:poisson con't better}, provided $Q$ is microsupported where $\vl_\pm = -1/2$, so that $Q \colon H^{s, \vl_{\min}} \lra H^{s, -1/2}$.
Proposition~\ref{prop:BP}, a crucial step in proving our main theorem, improves on this; it
asserts that increased regularity of the boundary distribution $f$
corresponds to increased decay of $P(\pm \lambda) f$ away from the radial sets.

The Poisson operator extends to a map on distributions, or equivalently on negative order Sobolev spaces $H^k(\partial M)$ for all $k \le 0$. The representation of the Poisson kernel as a distribution associated to a intersecting pair of Legendre submanifolds with conic points in \cite{mezw96} shows that there is a mapping property of the form 
\begin{equation}
P(\pm \lambda) : H^k(\partial M) \to H^{s, k + c}(M) \text{ for all } k \le 0
\label{eq:Pondist}\end{equation}
for  some $c \in\RR$;  this is obtained by passing derivatives from the distribution onto the Poisson kernel, thus reducing to the case $k \geq 0$. The next proposition, plus positive commutator estimates as in \cite[Section 3]{ghsz20}, show that in fact $c = -1/2$ is sharp for \eqref{eq:Pondist}, although we do not need this fact in the present article. 

\begin{proposition}\label{prop:BP}
  Let $Q \in \Psi_{\emph{sc}}^{0, 0}$ satisfy $\WF'(Q) \cap (\mathcal{R}_+
  \cup \mathcal{R}_-) = \varnothing$.  Then for all $s \in
  \mathbb{R}$, $k \in \mathbb{R}$,
  \begin{equation}
    \label{eq:decay poisson non-radial}
    Q P(\pm \lambda) \colon H^{k}(\partial M) \lra H^{s,  k - 1/2}(M).
  \end{equation}
\end{proposition}
Before we prove the proposition we note the straightforward corollary.
\begin{corollary}\label{thm:cor}
  Let $A_+$ be chosen as in \eqref{eq:main commutator}, and $s$ be any real number.  Then for $f
  \in H^k(\partial M)$,
\begin{equation}
[H, A_\pm] P( \lambda) f, \, [H, A_\pm] P( -\lambda) f  \in H^{s, k + 1/2}(M)
\label{comm}\end{equation}
and is microlocally trivial in a neighborhood of $\mathcal{R}_+ \cup
\mathcal{R}_-$. Consequently we also have 
\begin{equation}
[H, A_\pm] P( \lambda) f, \, [H, A_\pm] P( -\lambda) f  \in \Y_\pm^{s, \vl_{\max} + 1; k, 0}(M).
\label{commY}\end{equation}
\end{corollary}
The statement about microlocal triviality here means that there is a
neighborhood $U$ of $\mathcal{R}_- \cup \mathcal{R}_+$ so that for any $Q
\in \Psisc^{0,0}$ with $\WF' (Q) \subset U$, $Q [H, A_\pm] P(\pm \lambda) f
\in H^{S, L}$ for any $S, L$, i.e.\ is rapidly decaying.

\begin{proof}[Proof of Corollary]The first statement \eqref{comm} follows immediately from Proposition~\ref{prop:BP} and the fact that $[H, A_\pm]$
are scattering pseudodifferential operators of order $(1, -1)$. To obtain \eqref{commY}, we first replace $s$ by $s+k$ in \eqref{comm}, and then note that the $k$th power of the module $\mathcal{M}_\pm$ maps the functions $[H, A_\pm] P( \pm \lambda) f$ to $H^{s, 1/2}(M)$ simply using the fact that module elements have order $(1,1)$, whence
we find that these functions are in $\Y_\pm^{s, 1/2; k, 0}(M)$. Finally, the microlocal triviality near the radial sets allows us
to vary the spatial order $1/2$ arbitrarily near the radial sets, so we can replace the order $1/2$ with $\vl_{\max}+1$. 
\end{proof}

We remark that when $H = H_0 := \Delta_{\mathbb{R}^n} $ is the flat
structure and $k \in \NN$, Proposition \ref{prop:BP} follows without using the Poisson kernel used in the more general proof below.
Let $P_0(\lambda)$ denote the Poisson operator in the flat case.  Assuming without
  loss of generality that $\supp Q \subset \{ r \ge 1 \}$, one can
  make use of the generators of rotation on $\mathbb{S}^{n-1},$ which
  form a family $\Rot := \{ V \}$ of vector fields commuting with
  $\Delta_{\mathbb{S}^{n-1}}$ and generating the cotangent space at
  every point, in particular $f \in H^k(\mathbb{S}^{n-1}) \iff \Rot^j f \subset
  L^2(\mathbb{S}^{n-1})$ for all $j \le k$.  Let $V$ denote an arbitrary such vector field.   Using $[H_0,
  V] = 0$ in $r > 0$, we claim that for $f \in H^k$, 
  \begin{equation}
    \label{eq:Poisson commute}
     P_0(\lambda) V f = V P_0(\lambda) f.
  \end{equation}
 This follows from the explicit formula for the Poisson
  operator and integration by parts, but in using the results
  presented here one can argue as follows.  For $f \in
  C^\infty(\mathbb{S}^{n-1})$, the expansion for $P_0(\lambda) f$ gives
  that $P_0(\lambda) V f - V
  P_0(\lambda) f$ lies in $\mathcal{X}^{s,\vl_+}$, but we have $(H_0 - \lambda^2) ( P_0(\lambda) V f - V
  P_0(\lambda) f) = 0$, so since $H_0 - \lambda^2$ is injective on this
  domain, \eqref{eq:Poisson commute} follows.  The condition on $\WF'(Q)$, means that $p := (\hat z, \nu, \mu) \in
  \WF'(Q)$ then $\mu \neq 0$, and thus we can find $V$ with
  $\sigma_{1,1}(V)(p) \neq 0$.  Since $QV \in \Psi^{-\infty, 1}$ is
  elliptic at $p$, if $f \in H^k$, since $Q P_0(\lambda) V f \in H^{s, -1/2}$ we have by
  scattering elliptic regularity that $Q P_0(\lambda) f \in H^{s, - 1/2 +
    k}$.  We will use the structure of the Schwartz kernel
  of the Poisson operator to deduce this same result for more general
  Hamiltonians.

  \begin{proof}[Proof of Proposition \ref{prop:BP}]
We prove \eqref{eq:decay poisson non-radial} for $+\lambda$ only.  This will follow by using the well understood structure of the integral
kernel of $P(\lambda)$, due in this generality to Melrose and Zworski
\cite{mezw96}. The proof of \eqref{eq:decay poisson non-radial} for $-\lambda$ follows analogously. 

 As we show below, using propagation of singularities,
it will suffice to restrict the microsupport of $Q$ to punctured
neighborhoods of $\mathcal{R}_-$ over small balls on the boundary, as follows.  First
pick a small coordinate patch $V$ on $\partial M$ with coordinates
$y$, and then a small neighborhood $V'$ of $\mathcal{R}_- \cap
{}^{\sca} T^*_V M$; our $Q$ will be supported in the punctured
neighborhood $U = V' \setminus \mathcal{R}_-$.  Concretely, this can
be parametrized as $\{ (x , y, \nu, \mu ) : x, \nu^2 + |\mu|_h^2 - \lambda^2 < \epsilon,
0 < |\nu + \lambda| < \epsilon \}$ with $y$ in the coordinate patch.

For such $Q$, the operator $QP(\lambda)$ takes the form \cite{mezw96}
\begin{equation}
e^{-i\lambda \cos d_h(y,y')/x} \tilde a(x, y, y') + e(x, y, y')
\label{Poisson-explicit}\end{equation}
where $d_h$ is the distance function on $(\partial M, h)$, $\tilde a$ is
smooth and is supported where $y$ is in a deleted
neighborhood of $y'$ (with $y'$ varying in the same chosen coordinate
patch as $y$) and we integrate with respect to the $h$-Riemannian measure on $\partial M$. Here $e$ is a smooth function which is rapidly
decaying as $x \to 0$ which we subsequently ignore.

Since $\mathcal{F} (H^{s,k}(\mathbb{R}^n)) = H^{k, s}(\mathbb{R}^n)$,
were we working in Euclidean space, it would now suffice to prove that
$\mathcal{F} \circ Q P(\lambda) \colon H^k \lra H^{k-1/2, s}$.  To
use this approach, we identify the open set $(0, \epsilon)_x \times V \subset
M^\circ$ with a subset of Euclidean space by first choosing any local
diffeomorphism $\psi: V \to \US^{n-1}$, and then writing $\hat{z} = \psi(y)$, $r = 1/x$, $z = r\hat z(y) \in \R^n$.  Then
$\mathcal{F} \circ Q P(\lambda)$ has Schwartz kernel
$$
\tilde K(\zeta, y') := \int e^{-iz \cdot \zeta} e^{-i\lambda |z|  \cos d_h(\hat z,\psi(y'))} a(z, y') \, dz.
$$
where $a(z, y) = \tilde a(x, y, y')$ is a classical symbol in the $z$-variable of
order $0$, smooth in $y$. We claim that for $\chi \in
C^\infty(\mathbb{R}^n)$ with $\chi(\zeta) = 1$ on $|\zeta| < 2
\lambda$ and \\ $\supp \chi \subset \{ |\zeta| < 4\lambda \}$, then
$$
K(\zeta, y') := \int e^{-iz \cdot \zeta} e^{-i\lambda |z|  \cos
  d_h(\hat z,\psi(y'))} a(z, y') \chi(\zeta) \, dz, \quad a \in S^0
$$
satisfies
\begin{equation}
\tilde K(\bullet, y') - K(\bullet, y')  \in
\mathcal{S}(\mathbb{R}^n_\zeta),\label{eq:difference}
\end{equation}
smoothly in $y'$, i.e. it is rapidly
decreasing as $\zeta \to \infty$, as follows from non-stationary
phase.

We now view $K$ as a Fourier integral operator mapping from the sphere
to the dual $\RR^n$, parametrized by $\zeta$, with $z$ playing the role of a
homogeneous phase variable. 
We use the mapping properties of FIOs found in
Chapter 25 of \cite{Hvol4}, in particular Theorem 25.3.8.  Recall
throughout that $a$ is supported where $y$ is close to $y'$ but not
equal to it.  Here $K$ is an FIO of order $0 -
((n-1) + n - 2n)/4 = 1/4$ 
with phase function $\phi(\zeta, z, y')  =-iz \cdot \zeta -i\lambda |z|  \cos
  d_h(\hat z,y')$ where $z$
is the auxiliary variable.  We compute that $\phi$ defines the
canonical relation
$$
C := \{ (\zeta, z; y', \mu') \mid \zeta = \lambda \hat z \cos d_h(\hat z,y')
+ \lambda \theta_2 \sin d_h(\hat z,y'),\quad  \mu' = \lambda |z| \theta_1 \sin d_h(\hat z,y') \}
$$
where $\theta_1$ and $\theta_2$ are the initial, resp. final, fibre
coordinates of the unit length geodesic between $y'$ and $\hat z$ with
respect to the Riemannian metric $h$ on $\partial M$,
interpreted as being in $T^*_{y'} \partial M$ in the first case and
$T^*_{y(\hat z)} \partial M$ in the second.  The dimension of the
canonical relation is $2n - 1$. It is parametrized by $z, y'$
since for $y, y'$ sufficiently close (recall that $y$ parametrizes
$\hat z$), the $\theta_i$ are determined by $y, y'$ and therefore by
$\hat z$ and $y$ and the other variables $\mu', \zeta$ are given in
terms of these.  On the other hand the projection of $C$ to the
`right' factor $T_{y'}^* \partial M$ is a surjective submersion
since $\partial \mu' / \partial \hat z$ is full rank for $y, y'$
sufficiently close but not equal.  Thus the pull back of the
symplective form on $T^* \partial M_{y'}$ to $C$ is rank $2(n - 1)$ or
corank $1$.  Finally, at no point are either of the  radial vector
fields $\mu' D_{\mu '}$
in $T^*\partial M \setminus \{ 0 \}$ nor  $z D_z$ in $T^* \mathbb{R}^n_\zeta
\setminus \{ 0 \}$ tangential to $C$.

Therefore \cite[Theorem 25.3.8]{Hvol4} together with the standard
argument (see e.g.\ \cite[Corollary 25.3.2]{Hvol4}) in which one
pre/post composes an FIO by invertible elliptic pseudodifferential
operators, give that $K$ maps $H^k(\partial M)$ to
$H^{k - 1/2}(\RR^n_\zeta)$ continuously.  (Note that the theorem concludes a
mapping to $H^{k - 1/2}_{\text{loc}}(\RR^n_\zeta)$ but $K$ is compactly
supported in $\zeta$.)  Compact support of $\tilde K$ in $\zeta$ and
\eqref{eq:difference} then give that $\langle \zeta \rangle^s \tilde K$
maps $H^k(\partial M)$ to $H^{k - 1/2, s}(\RR^n_\zeta)$ for any $s$,
whence composing with the inverse Fourier transform gives
\eqref{eq:first improvement} for $Q$ microlocalized near the radial
sets and locally over the boundary.

To treat general $Q$ with $\WF'(Q) \cap (\mathcal{R}_+ \cup
\mathcal{R}_-) = \varnothing$, we recall that the assertion of continuity in
\eqref{eq:decay poisson non-radial} can be microlocalized in the sense
that it will follow if, for all $q \in \WF'(Q)$ there is $\tilde Q \in \Psisc^{0,0}$
such that $q \in \Ell_{0,0}(\tilde Q)$ and for some $C > 0$
$$
\| \tilde Q P(\lambda) f \|_{H^{s, k -1/2}(M)} \le C  \| f
  \|_{H^k(\partial M)}.
  $$
  for all $f \in H^m(\partial M)$.  
For $q \in \Ell_{2, 0}(H - \lambda^2)$ this follows from
the microlocal elliptic estimate (see e.g.\ \cite[Prop.\ 3.3]{ghsz20}
\begin{equation*}
  \begin{gathered}
    \| \tilde Q P(\lambda) f \|_{H^{s, k -1/2}(M)} \lesssim \| (H -
    \lambda^2) P(\lambda) f \|_{H^{s -2, k- 1/2}(M)} + \| P(\lambda)
    f \|_{H^{- S, -L}}) \\
    = \| P(\lambda) f \|_{H^{- S, -L}(M)} \lesssim \| f \|_{H^k(\partial M)}
  \end{gathered}
  \end{equation*}
where $S, L$ are chosen sufficiently large and the final bound comes from \eqref{eq:Pondist}.  (Here $\lesssim$ means
there exists a $C > 0$ such that the LHS is bounded by $C$ times the RHS.)
For $q \in
\Sigma_{2,0}(H - \lambda^2) \setminus (\mathcal{R}_+ \cup \mathcal{R}_-)$, then there is a
bicharacteristic ray $\gamma$ in $\Sigma$ such that $\gamma(0) = q$
and $\lim_{\sigma \to - \infty} \gamma(\sigma) \in \mathcal{R}_-$.  The first
part of our proof implies the existence of a $\sigma_0 << 0 $ and $Q
\in \Psisc^{0,0}$ such that $\gamma(\sigma_0) \in \Ell_{0,0}(Q)$ and $Q$
satisfies \eqref{eq:decay poisson non-radial}.  On the other hand, by the
standard propagation of singularities estimate (see e.g. \cite[Prop.\
3.4]{ghsz20}), there exists $\tilde Q \in \Psisc^{0,0}$ with $q \in \Ell_{0,0}(\tilde Q)$ such that
\begin{equation*}
  \begin{gathered}
    \| \tilde Q P(\lambda) f \|_{H^{s, k -1/2}(M)} \phantom{  \| \tilde Q P(\lambda) f \|_{H^{s, k -1/2}(M)}  
    \| \tilde Q P(\lambda) f \|_{H^{s, k -1/2}(M)}  \| \tilde Q P(\lambda) f \| }
    \\ \lesssim \| Q
    P(\lambda) f \|_{H^{s, k -1/2}(M)}  + \| (H -
    \lambda^2) P(\lambda) f \|_{H^{s -2, k + 1/2}(M)} + \| P(\lambda)
    f \|_{H^{- S, -L}}) \\
    = \| Q
    P(\lambda) f \|_{H^{s, k -1/2}(M)} + \| P(\lambda) f \|_{H^{- S, -L}} \\ \lesssim \| Q
    P(\lambda) f \|_{H^{s, k -1/2}(M)} + \| f \|_{H^k}
  \end{gathered}
\end{equation*}
where we choose $S, L$ sufficiently large and use \eqref{eq:Pondist} to estimate $P(\lambda) f$ by $f$. 
Then,  using $\| Q P(\lambda) f \|_{H^{s, k -1/2}(M)} \lesssim  \| f \|_{H^{k}}$ gives the result for $q$ on the characteristic set away from the radial
sets, completing the proof.    \end{proof}

The next proposition is known to experts, see e.g. \cite{Va98} and \cite[Section 3]{HMV04} in a slightly different context. 
It is implicit in older works, e.g. \cite{IK84}; no doubt many other references could be given. 
Notice that in \eqref{u-decomp},
$R(\lambda + i0) - R(\lambda - i0)$ is, up to a constant, the `spectral projection' $dE(\lambda^2)$ for
the operator $H$. This operator is not an actual projection, however; in fact, it is not bounded on any
natural Hilbert space. The proposition shows that a modification of this operator leads to a genuine projection,
that acts as the identity on the generalized eigenfunctions. 

\begin{proposition}\label{prop:L2} 
Let $u = P(\lambda)f$, for $f \in H^k(\partial M)$, $k = 0, 1, 2, \dots$, and let $s \in \mathbb{R}$. Then  we can express
\begin{equation}\begin{gathered}\label{u-decomp}
u = u_+ + u_-, \\ 
u_\pm = A_\pm u = R(\lambda \pm i0) [H, A_\pm] u \in \X_\pm^{s+2, \sw_\pm; k, 0}(M). 
\end{gathered}\end{equation}

Conversely, let $w \in H^{s, k+1/2}(M)$ be such that $w$ is microlocally trivial  in a neighbourhood of $\Rm \cup \Rp$. 
Then, defining $u$ by
$$
u := \Big( R(\lambda + i0) - R(\lambda - i0) \Big) w,
$$
$u$ can be obtained as $u = P(\lambda)f$ for a function $f \in H^k(\mathbb{S}^{n-1})$, namely the function $f = (i/2\lambda)P(\lambda)^* w$. 
\end{proposition}

\begin{proof}
First we show \eqref{u-decomp}. Using the identity $(H - \lambda^2) u = 0$ we evaluate the right hand side.
\begin{align*}
  &\Big( R(\lambda + i0) - R(\lambda - i0) \Big) [H, A_+] u \\ &= R(\lambda
+ i0) [H - \lambda^2, A_+] u + R(\lambda- i0) [H -\lambda^2, A_-] u
  \\
  & =  R(\lambda + i0) (H - \lambda^2) A_+ u + R(\lambda - i0) (H - \lambda^2) A_- u. 
\end{align*}

We claim that $A_+ u$ is in the space $\X^{s, \sw_+}$ (for any $s$). To see this, we use \eqref{eq:poisson con't},
which shows that $P(\lambda) f$ is in $H^{s, -1/2 - \epsilon}$ globally, for arbitrary $\epsilon > 0$, and
Proposition~\ref{prop:BP}, which shows that it is in $H^{s, -1/2}$ microlocally away
from $\Rp$, together with the definition of the operator $A_+$ which microlocalizes away from
$\Rm$. Similarly, $A_- u$ is in the space $\X^{s, \sw_-}$. 
On these spaces,
$R(\lambda + i0)$, respectively $R(\lambda - i0)$ is a left inverse to
$H - \lambda^2$ (see \eqref{eq:var spaces}). It follows that the RHS is equal to $A_+ u + A_- u =
u$, proving the identity \eqref{u-decomp} with $u_\pm = A_\pm u$. Moreover, by Corollary~\ref{thm:cor}, $[H, A_\pm]u$ is in $\Y_\pm^{s, \sw_{\max} +1; k, 0} \subset \Y_\pm^{s, \sw_\pm +1; k, 0}$,  so using the resolvent mapping property \eqref{eq:var spaces modreg}, we see that $u_\pm$ is in $\X_\pm^{s+2, \sw_\pm; k, 0}(M)$ as claimed.  

We show the converse statement first for $k=0$. This  follows immediately from \eqref{sm} and \eqref{eq:poisson
  con't better}  since the assumption
on $w$ implies that $w \in H^{s, \vl_{\text{max} }+1}$, provided that $\delta$ in the definition of $\sw_\pm$ is sufficiently small, so that $w$ is microlocally trivial in the region where $\sw_{\max} \neq -1/2$. For any positive integer $k$, we
write 
$$P(\lambda)^* w = P(\lambda)^* Qw + P(\lambda)^* (\Id - Q) w,$$
where $Q$ is microlocally equal to the identity on the microsupport of $w$, and microlocally trivial near the radial sets.
Then $(\Id - Q) w$ is Schwartz, so using \eqref{eq:Padsmooth}, $P(\lambda)^* (\Id - Q) w$ is $C^\infty$. For the other term,
the adjoint of Proposition~\ref{prop:BP} shows that $P(\lambda)^* Qf$ is in $H^k(\partial M)$. 
\end{proof}

We now begin our analysis of the incoming/outgoing data of distributions
in the image of $R(\lambda \pm i 0)$.  The main improvement in the
regularity of this incoming/outgoing data comes from a combination of the
mapping properties of the Poisson operator and its adjoint, together
with the reproducing formula \eqref{u-decomp}.

\begin{proposition}\label{prop:reslimits} 
Let $s \geq 0$, and assume $\wmac$ is in $\Y_+^{s, \vl_{\max} +1; 1, k-1}$, $k \geq 2$, where $\vl_{\max}$ is defined in \eqref{eq:l min l max}. Assume that $\delta$ satisfies
\begin{equation}
0 < 2\delta < \min(1, \gamma - 1),
\label{eq:eps-cond}\end{equation}
where $\gamma$ is as in \eqref{Vests}. 
    Then $u_+ := R(\lambda + i0) \wmac$ is such that, in a collar neighbourhood of the boundary $\partial M \times (R, \infty)_r$, the limit  
$$ 
\mathcal{L}(\lambda) u_+ := \lim_{r \to \infty} r^{(n-1)/2} e^{- i\lambda r}  u_+(r, \cdot) 
$$
exists in $H^{k-2}(\partial M)$, with estimates 
\begin{equation}\begin{gathered}
\Big\| \mathcal{L}(\lambda) u_+ \Big\|_{H^{k-2}(\partial M)} \leq C \| \wmac \|_{\Y_+^{s, \vl_{\max} +1; 1, k-1}}, \\
\big\| r^{(n-1)/2} e^{- i\lambda r}  u_+(r, \cdot) - \mathcal{L}(\lambda) u_+ \big\|_{H^{k-2}(\partial M)} = O(r^{-\epsilon})
\end{gathered}\label{Lm}\end{equation}
for $0 < \epsilon < \delta $. Similarly, if $\wmac \in \Y_-^{s, \vl_{\max} +1; 1, k-1}$, $k \geq 2$, then $u_- := R(\lambda - i0)
\wmac$ is such that the limit
$$ 
\mathcal{L}(-\lambda) u_- := \lim_{r \to \infty} r^{(n-1)/2} e^{+ i\lambda r}  u_-(r, \cdot) 
$$
exists in $H^{k-2}(\partial M)$, with estimates 
\begin{equation}\begin{gathered}
\Big\| \mathcal{L}(-\lambda) u_- \Big\|_{H^{k-2}(\partial M)} \leq C \| \wmac \|_{\Y_-^{s, \vl_{\max} +1; 1, k-1}}, \\
\big\| r^{(n-1)/2} e^{+ i\lambda r}  u_-(r, \cdot) - \mathcal{L}(-\lambda) u_- \big\|_{H^{k-2}(\partial M)} = O(r^{-\epsilon})
\end{gathered}\label{Lp}\end{equation}
for $0 < \epsilon < \delta $. 
\end{proposition}


\begin{proof}
We follow the argument of Section 4 of \cite{ghsz20}, which in turn is based on \cite{me94}. We do this only for $u_+$, as the argument is similar for $u_-$. From now on we denote $u_+$ by $u$. 

We use the microlocalizing operators $Q_\pm$ from Section~\ref{sec:micro}, and complete these to a partition of unity, $\Id = Q_+ + Q_- + Q_3$. Thus, $Q_+$ is microsupported close to $\Rp$, $Q_-$ is microsupported close to $\Rm$, and $Q_3$ is microsupported away from $\Rp \cup \Rm$. Assuming that the microsupports of $Q_\pm$ are sufficiently close to $\mathcal{R}_\pm$, the assumption on $F$ implies that $Q_\pm F$ is in the space $\Y_+^{s, 1/2 + \delta; 1, k-1}$, while $Q_3 F$ is in $\Y_+^{s, 1/2 ; 1, k-1}$. 

We notice that all of the commutators $[H, Q_\bullet]$ have boundary order $1$ and microlocally vanish near $\Rp \cup \Rm$, so $r^2 [H, Q_\bullet]$ belongs in both modules, $\mathcal{M}_+$ and $\mathcal{M}_-$.  We write $u_1 = Q_+ u$, $u_2 = Q_- u$, and $u_3 = Q_3 u$.

Write $\utilde_1 = \chi(r) r^{(n-1)/2} e^{-i\lambda r} u_1$, where $\chi$ is supported in $r > R$ and identically $1$ near $r \geq 2R$. 
Our first goal is to show that $\utilde_1(r,y)$ has a limit $b(y)$ in $H^{k-2}(\partial M)$, and that $\utilde_1(r,y) - b(y) = O_{H^{k-2}(\partial M)}(r^{-\epsilon})$ as $r \to \infty$. 
To do this, we write the operator $H - \lambda^2$ in the form 
\begin{equation}\label{ueqn2}
H - \lambda^2 =  D_r^2  - i (n-1)r^{-1} D_r + r^{-2}
L + r^{-2} \tilde L +V - \lambda^2, 
\end{equation}
where $L$ is a second order differential operator involving only tangential $D_{y_j}$ derivatives, and $\tilde L$ is a scattering differential operator of order $(1,0)$. Since $(H - \lambda^2) u = \wmac$,  we obtain 
\begin{multline}\label{ueqn3}
\Big( D_r + \lambda \Big) \Big( D_r -  \lambda  - \frac{i (n-1)}{2r} \Big) u_1  =  Q_+ \wmac + [H, Q_+] u - r^{-2} \tilde L u_1 \\ + \frac{1}{2r^2} \Big( i (n-1) r(D_r - \lambda)  +(n-1)   - 2L    \Big)u_1 \,  - V u_1  .
\end{multline}
We are going to show that the RHS lies in $\Y^{s, 1/2 + \delta; 0, k-2}$, with the norm in this space bounded by 
\begin{equation}
C  \| \wmac \|_{\Y_+^{s, \vl_{\max} +1; 1, k-1}}.
\label{Fnorm}\end{equation}
This has already been noted for the first term, $Q_+ \wmac$. For the remaining terms, it suffices to bound them by 
\begin{equation}
C \| u \|_{\Y^{s+2, -1/2 - \delta; 1, k-1}}
\label{unorm}\end{equation}
since, as we have seen in \eqref{eq:var spaces modreg}, this is bounded by \eqref{Fnorm}. For the second term, we use the fact that $[H, Q_+]$ is $r^{-2}$ times a small module element, so the $\Y^{s, 1/2 + \delta; 0, k-2}$-norm of this term is bounded by the $\Y^{s+1, -3/2 + \delta; 0, k-1}$ norm of $u$, which is bounded by \eqref{unorm} since $\delta < 1/2$ according to  \eqref{eq:eps-cond}. The next term is very similar: since $r^{-2} \tilde L$ has order $(1, -2)$, the $\Y^{s, 1/2 + \delta; 0, k-2}$-norm of this term is bounded by the $\Y^{s+1, -3/2 + \delta; 0, k-2}$ norm of $u$, which is again bounded by \eqref{unorm}. 
For the term with the $r^{-2}$ prefactor, we observe that 
the differential operator in large parentheses is contained in $\mathcal{M}_+ \cdot \mathcal{N}$, that is, one factor in the large module and one in the small module. It follows that the $\Y^{s, 1/2 + \delta; 0, k-2}$-norm of this term is bounded by the $\Y^{s, -3/2 + \delta; 1, k-1}$ norm of $u_1$, which again is bounded by \eqref{unorm}. Finally, for the $V$ term this follows from the fact that $V$ satisfies the conormal estimates \eqref{Vests}, hence multiplication by $V$ is a scattering pseudodifferential operator of order $(0, -\gamma)$. Recalling \eqref{eq:eps-cond}, we have $-1/2 - \delta + \gamma > 1/2 + \delta$, and it follows that the $\Y^{s, 1/2 + \delta; 0, k-2}$-norm of $Vu_1$ is bounded by the $\Y^{s, -1/2 - \delta; 0, k-2}$-norm of $u_1$, which is bounded by \eqref{unorm}.


Now observe  the operator $D_r + \lambda$ is elliptic everywhere on $\WF'(Q_+)$, since $Q_+$ is microsupported near $\Rp$.  Thus we may write invert this operator microlocally; that is, we can write 
$$
\Id = J (D_r + \lambda) + R',
$$
where $J \in \Psisc^{-1, 0}$ is a microlocal inverse, and the microsupport of the remainder $R'$ is disjoint from $Q_+$. Then for any scattering pseudodifferential operator $A$, $R' A Q_+ u$ can be bounded by (a suitable multiple of) any Sobolev norm of $u$, for example $\| u \|_{s+2, -1/2 - \delta}$,  which in turn is bounded by \eqref{unorm}. 
Applying this operator identity to $A Q_+ u = A u_1$, where $A = (D_r -  \lambda  - i (n-1)/(2r))$, then applying \eqref{ueqn3} and the estimate \eqref{Fnorm} on the RHS of this equation, we find 
\begin{equation}\label{ueqn4}\begin{gathered}
 \Big(D_r -  \lambda - \frac{i(n-1)}{2r}  \Big) u_1  \in \Y^{s, 1/2 + \delta; 0, k-2}, \\
 \Big\|  \Big(D_r -  \lambda - \frac{i(n-1)}{2r}  \Big) u_1 \Big\|_{\Y^{s, 1/2 + \delta; 0, k-2}} \leq C \| \wmac \|_{\Y_+^{s, \vl_{\max} +1; 1, k-1}}.
\end{gathered}\end{equation}
Now, using $s \ge 0$ and observing that 
$$
D_r \utilde_1 = \chi (r) r^{(n-1)/2} e^{-i\lambda r} \Big( D_r -  \lambda -  \frac{i(n-1)}{2r}  \Big) u_1 + (D_r \chi) r^{(n-1)/2} e^{-i\lambda r}  u_1,
$$
we find that $D_r \utilde_1 \in H^{0, 1/2 + \delta - (n-1)/2; 0, k-2}$, or equivalently
$$
D_r \utilde_1 \in  r^{n/2 - 1 - \delta} L^2 \big( ([R, \infty), r^{n-1} dr); H^{k-2}(\partial M) \big),
$$
with a corresponding norm estimate (where we used the support property of $D_r \chi$ for the inclusion in $H^{0, 1/2 + \delta -(n-1)/2; 0, k}$ of the second term). Combining this estimate with the inclusions

\begin{equation} \label{L one inclusion} \begin{gathered}
 r^{n/2 - 1 -\delta} L^2 \big( ([R, \infty), r^{n-1} dr); H^{k-2}(\partial M) \big) \subseteq  r^{-1/2 - \delta} L^2\big( ([R, \infty),  dr); H^{k-2}(\partial M) \big) \\ \subseteq r^{-\epsilon} 
L^1\big( ([R, \infty),  dr); H^{k-2}(\partial M) \big), \qquad 0 < \epsilon < \delta,
\end{gathered}\end{equation}
we find 
\begin{equation} \label{L one inclusion estimate}
\Big\| r^{\epsilon} D_r \utilde_1 \Big\|_{L^1\big( ([R, \infty),  dr); H^{k-2}(\partial M) \big) } \leq C \| \wmac \|_{\Y_+^{s, \vl_{\max} +1; 1, k-1}}.
\end{equation}
We note that, since $\utilde_1$ is locally $H^1$ in $r$ with values in $H^{k-2}(\partial M)$, it is in fact continuous in $r$ with values in $H^{k-2}(\partial M)$. By \eqref{L one inclusion estimate}, we can integrate to infinity to find
$$
b(y) =  \int_R^\infty \partial_{r'} \utilde_1(r', y) \, dr', \quad \| b \|_{H^{k-2}(\partial M)} \leq C \| \wmac \|_{\Y_+^{s, \vl_{\max} +1; 1, k-1}},
$$
is well-defined as an element of $H^{k-2}(\partial M)$. 
Moreover, 
\begin{equation}\begin{gathered}
\utilde_1(r, y) - b(y) = -\int_r^\infty \partial_{r'}  \utilde_1(r', y) \, dr' , \\
\Big\| \utilde_1(r, y) - b(y) \Big\|_{H^{k-2}(\partial M)} \leq C r^{-\epsilon} \Big\| r^{\epsilon} D_r \utilde_1 \Big\|_{L^1\big( ([R, \infty),  dr); H^{k-2}(\partial M) \big) } \\ \leq C r^{-\epsilon}  \| \wmac \|_{\Y_+^{s, \vl_{\max} +1; 1, k-1}}.
\end{gathered}\end{equation}

To treat $u_2$, we apply a similar argument with the role of $\Rp$ and $\Rm$ interchanged. In this case we define $\utilde_2 = \chi(r) r^{(n-1)/2} e^{+i\lambda r} u_2$ and write \eqref{ueqn3} in the form 
\begin{multline}\label{ueqn33}
\Big( D_r - \lambda \Big) \Big( D_r +  \lambda  - \frac{i (n-1)}{2r} \Big) u_2 \\ = Q_- \wmac + [H, Q_-] u_2 + \frac{i (n-1)}{2r^2} (r(D_r + \lambda)) u_+ + \Big(\frac{n-1}{2r^2}  - r^{-2} L  - r^{-2} \tilde L - V \Big) u_+  .
\end{multline}
Folllowing the same reasoning as above, we find that $\utilde_2$ has a limit as $r \to \infty$, with the same $O(r^{-\epsilon})$ of convergence. But  here, the mapping property of the \emph{outgoing} resolvent near the \emph{incoming} radial set (the inverse mapping to \eqref{eq:var spaces modreg}, see \cite[Theorem 3.1]{ghsz20}) shows that $u_2$ is actually in $H^{s, -1/2 +\delta}$, i.e.\ the spatial order is above threshold, since $\vl_+ = -1/2 + \delta$ on the microsupport of $Q_-$. Therefore, 
since we know that $\utilde_2$ has a limit in $H^{k-2}(\partial M)$ as $r \to \infty$, this limit must be zero (were the limit nonzero, then $\tilde{u}_2$ would fail to lie in the space $H^{s, -1/2 +\delta -(n-1)/2}$ for $\delta > 0$). It then follows that $\chi(r) r^{(n-1)/2} e^{-i\lambda r} u_2 = e^{-2i\lambda r} \utilde_2$ also has a zero limit in $H^{k-2}(\partial M)$ as $r \to \infty$, with the same rate of convergence as $\tilde{u}_2$.

It remains to discuss $u_3$. 
We claim that $u_3$ is an element of $\Y_+^{s+2, 1/2 -\delta; 1, k-2}$. We argue separately in the microlocal regions (i) near the characteristic variety $\Sigma$, and (ii) away from $\Sigma$. The mapping property \eqref{eq:var spaces modreg} shows that $u$ is in $\Y_+^{s+2, -1/2 -\delta; 1, k-1}$.

In region (i), since $u_3$ is microsupported away from the radial sets, the small module $\mathcal{N}$ is elliptic there, so we can trade one order of small module regularity for a gain of one spatial order. In region (ii), we already have $u \in \Y_+^{s+2, +1/2 -\delta; 1, k-1}$ since $H - \lambda^2$ is elliptic there and we incur no loss in the spatial regularity from applying the resolvent in this region. 

Now using the $\mathcal{M}_+$ module regularity, and replacing the differential order with zero, we see that 
$$
r\big(D_r - \lambda - i\frac{n-1}{2r} \big) u_3 \in H^{0, 1/2 - \delta; 0, k-2},
$$
that is, 
$$
\big(D_r - \lambda - i\frac{n-1}{2r} \big) u_3 \in H^{0, 3/2 - \delta; 0, k-2},
$$
which is stronger than \eqref{ueqn4} as $\delta < 1/2$. We can thus apply the same reasoning as for $u_1$ to obtain a limits
for $\chi(r) r^{(n-1)/2} e^{\pm i\lambda r} u_3$ (which are necessarily
zero, for the same reason as for $u_2$), with the same rate of convergence.

The estimates in \eqref{Lm} are obtained by adding the contributions from $u_1$, $u_2$ and $u_3$. The estimates \eqref{Lp} are obtained similarly.

\end{proof}

\begin{remark}\label{rem:zerolimits} From the proof above we see the following: if  $Q \in \Psisc^{0,0}$ is microsupported away from $\Rp$, then $\mathcal{L}(\lambda)(Q u_+) = 0$. Similarly,
if $Q' \in \Psisc^{0,0}$ is microsupported away from $\Rm$, then $\mathcal{L}(-\lambda)(Q' u_-) = 0$. Also note that $\mathcal{L}(\lambda)(Q u_+) = 0$ immediately implies $\mathcal{L}(-\lambda)(Q u_+) = 0$. In the same way, $\mathcal{L}(-\lambda)(Q' u_-) = 0$ implies $\mathcal{L}(\lambda)(Q' u_-) = 0$.
\end{remark}

 \begin{proposition}\label{prop:limit-reg} Let $f \in H^k(\partial M)$, with $k \in \NN$, $k \geq 2$, and let $u = P(\lambda) f$. Then $A_- u$ is such that the limit 
 \begin{equation}
\mathcal{L}(-\lambda)(A_- u) = \lim_{r \to \infty} r^{(n-1)/2} e^{i\lambda r}  (A_- u)(r, \cdot) 
 \label{Amuonsphere}\end{equation}
exists in $H^{k-2}(\partial M)$.  Moreover, this limit is $f$. Similarly, the limit 
 \begin{equation}
\mathcal{L}(\lambda)(A_+ u) =  \lim_{r \to \infty}  r^{(n-1)/2} e^{-i\lambda r}  (A_+ u)(r, \cdot) 
 \label{A+limit}\end{equation}
 exists in $H^{k-2}(\partial M)$ as $r \to \infty$. Moreover, this
 limit is in $H^k(\partial M)$. 
 Both limits are achieved with an $O(r^{-\epsilon})$ convergence rate, as in Proposition~\ref{prop:reslimits}. 
 \end{proposition}

 \begin{proof}
 As we have already seen in Proposition~\ref{prop:L2}, $A_-u$ is the incoming resolvent applied to $[H, A_-] u$, and $A_+u$ is the outgoing resolvent applied to $[H, A_+]u = - [H, A_-] u$. 
 We have also seen in Corollary~\ref{thm:cor}  that for $f \in H^k(\partial M)$,  $[H, A_\pm] u$ is in the module regularity space
 $\Y_\pm^{s, \vl_{\max} +1; k, 0}$.  Therefore, the existence of the limits \eqref{Amuonsphere} and \eqref{A+limit} in $H^{k-2}(\partial M)$ follows from  Proposition~\ref{prop:reslimits}. 

We next note that, for $f \in C^\infty(\partial M)$, and $u = P(\lambda) f$,  the limit $\mathcal{L}(-\lambda)(A_- u)$ is exactly $f$. This is a defining property of the Poisson kernel; see \cite[Equations (0.2), (0.4)]{mezw96}. However, we will elaborate on this point, as it is closely related to the form of \eqref{Poisson-explicit}. Suppose, without loss of generality, that $f$ is supported in a small neighbourhood of $b \in \partial M$. Outside any microlocal neighbourhood of $\Rm$, the contribution of $A_- u$ to this limit is zero, as it is in $H^{s, k}$ for every $k$ by Proposition~\ref{prop:BP}. Therefore, given the canonical relation of $P(\lambda)$, see \cite[Propositions 4 and 19]{mezw96}, we can replace $A_- P(\lambda) f$ by $Q P(\lambda) f$ where 
$Q$ is microsupported near $\Rm$ and its kernel is supported near $(b,b)$. The kernel of $Q P(\lambda)$ is then as in \eqref{Poisson-explicit}, that is, 
$$
e^{-i\lambda \cos d_h(y,y')/x} \tilde a(x, y, y') + e(x, y, y'),
$$
except that $\tilde a(x,y, y')$ is now supported close to $(0, b, b)$ but not in a \emph{deleted} neighbourhood, as was the case in \eqref{Poisson-explicit}. In fact we have 
$$\tilde a(0, y, y) = (\lambda/(2\pi))^{(n-1)/2} e^{-i(n-1)\pi/4} \text{ for $y$ near $b$. }$$
Then, in normal coordinates around $y$, we have 
\begin{equation*}\begin{gathered}
\cos d_h(y, y') = 1 - \frac{|y-y'|^2}{2} + O(|y-y'|^3), \\ 
dh(y') = dy'(1 + O(|y-y'|)),
\end{gathered}\end{equation*}
the stationary phase lemma shows that indeed $\mathcal{L}(-\lambda)(Q P(\lambda) f)(y) = f(y)$ in a neighbourhood of $b$. 
 
 Now for arbitrary $f \in H^k(\partial M)$, we choose a sequence of smooth $f_j$ converging to $f$ in $H^k(\partial M)$. Let $u_j = P(\lambda) f_j$; then $[H, A_-]u_j$ converges to $[H, A_-] u$ in $\Y_-^{s, \vl_{\max} +1; 1, k-1}$ using Corollary~\ref{thm:cor}. As above, we have $A_- u_j = R(\lambda - i0) [H, A_-] u_j$. 
 Using \eqref{Lm} in Proposition~\ref{prop:reslimits},  $\lim_{j \to \infty} \mathcal{L}(-\lambda) A_- u_j$ exists in $H^{k-2}(\partial M)$ and is $\mathcal{L}(-\lambda) A_- u$. On the other hand, since the $f_j$ are smooth, $\mathcal{L}(-\lambda) A_- u_j$ is precisely $f_j$, which converge to $f$ in $H^k$, and so \emph{a fortiori} in $H^{k-2}$, showing that $\mathcal{L}(-\lambda) A_- u = f$.

 To obtain the result for $A_+ u$, we could just appeal to the main result of \cite{mezw96} that says that the limit is the scattering matrix $S(\lambda)$ applied to $f$, and since $S(\lambda)$ is an FIO of order zero, then $S(\lambda) f$ is in $H^k$. However, we prefer a direct argument. Using the formula \eqref{sm}, we see that $u$ can be expressed as 
 $$
u = \frac1{2\lambda i} P(-\lambda) P(-\lambda)^* [H, A_+] u.  $$
 That is, $u$ is equal to $P(-\lambda) f'$, where $f' =  (2\lambda i)^{-1} P(-\lambda)^* [H, A_+] u$. Now, arguing as in the proof of the converse to Proposition~\ref{prop:L2}, but for $P(-\lambda)$ instead of $P(\lambda)$, we see that $f'$ is in $H^k(\partial M)$. 
Now applying the argument in the first half of this proof, with signs switched, we conclude that the limit \eqref{A+limit} exists in $H^{k-2}(\partial M)$ and is equal to $f' \in H^k(\partial M)$.

The statement about the convergence rate is shown by applying Proposition~\ref{prop:reslimits} to $F = [H, A_\pm] u$. 
\end{proof}

\begin{remark} The previous proof shows that the operators $\mathcal{L}(\pm \lambda) \circ R(\lambda \pm i0)$ in Proposition~\ref{prop:reslimits} coincide with $\pm (2i\lambda)^{-1} P(\mp \lambda)^*$. Also, we remark that $\mathcal{L}(\lambda) A_+ P(\lambda) f$, the limit in \eqref{A+limit}, is precisely $S_{\mathrm{lin}}(\lambda) f$, the linear scattering matrix applied to $f$. 
\end{remark}

\begin{remark}\label{rem:conv} There is a subtlety here: the limiting functions in \eqref{Amuonsphere} and \eqref{A+limit} are more regular than one would expect based on the topology of convergence.
 In fact, the convergence does \emph{not} take place, in general, in the topology of $H^k(\partial M)$. To see this, consider the operator that maps
 $f$ to $r_0^{(n-1)/2} P(\lambda) f$ restricted to $\{ r = r_0 \}$. This  is a semiclassical FIO of order zero on $\partial M$ (with $1/r_0$ playing the role of semiclassical parameter)
 but the canonical relation has fold singularities. This is best seen in the case of flat Euclidean space, where the phase function of this FIO 
 is $\Phi(\hat z, \omega) = - \lambda \hat z \cdot \omega$, $z = r \hat z$, $\hat z, \omega \in \mathbb{S}^{n-1}$.
 Such an FIO cannot be expected to be (and is not) bounded on $H^k$ uniformly in $r_0$, and hence,
 convergence cannot be expected to take place, even weakly, in $H^k$ as $r_0 \to \infty$. 
 \end{remark}

 We combine the previous two propositions to obtain
 \begin{proposition}\label{prop:reg}
 The limits in Proposition~\ref{prop:reslimits} lie in $H^k(\partial M)$. 
 \end{proposition}

 \begin{proof} In the notation of  Proposition~\ref{prop:reslimits}, we have $F \in \Y_\pm^{s, \vl_{\text{max}} + 1; 1, k-1}$, and therefore it lies in both of the variable order module regularity spaces $\Y^{s, \vl_\pm +1; 0, k}$. It follows that $u_\pm  \defeq R(\lambda \pm i0) F$ lies in $\X^{s+2, \vl_\pm; 0, k}$. In particular, $u \defeq u_+ - u_-$ is in $H^{s+2, k-1/2}$ microlocally away from small  neighbourhoods of $\Rp$ and $\Rm$, since $\vl_\pm = - 1/2$ there. (Notice that, on the characteristic set but away from the radial sets, the small module is elliptic, so small module regularity of order $k$ affords $k$ orders of spatial regularity. Away from the characteristic set, $u$ is rapidly decaying by microlocal ellipticity, since $(H - \lambda^2)u = 0$.) 
 
 We next observe that, by \eqref{sm},  $u = P(\lambda) f'$ with $f' = (2\lambda i)^{-1} P(\lambda)^* F$. Using \eqref{eq:poisson con't better} we see that $f' \in L^2$.  Applying Proposition~\ref{prop:L2}, with $k=0$, we have $u = (R(\lambda + i0) - R(\lambda - i0)) [H, A_+] u$. But then, since $[H, A_+] u$ is in $H^{s+1, k+1/2}$, and microlocalized away from $\Rp$ and $\Rm$, Proposition~\ref{prop:L2} shows that $u = P(\lambda) f$ with $f \in H^k(\partial M)$. 
 
Consider the limit $\mathcal{L}(\lambda)u_+$. We claim that is the same as $\mathcal{L}(\lambda) A_+ u$. In fact, the difference is 
\begin{equation*}\begin{gathered}
\mathcal{L}(\lambda)(u_+ - A_+ u) \\
= \mathcal{L}(\lambda)(u_+ - A_+ u_+ - A_+ u_-) \\
= \mathcal{L}(\lambda)(A_-u_+  - A_+ u_-).
\end{gathered}\end{equation*}
 Using Remark~\ref{rem:zerolimits}, we see that $\mathcal{L}(\lambda) A_- u_+ = 0$ and $\mathcal{L}(-\lambda) A_+ u_- = 0$, which implies that $\mathcal{L}(\lambda) A_+ u_- = 0$. Thus 
 $\mathcal{L}(\lambda)u_+=\mathcal{L}(\lambda) A_+ u$. Similarly, $\mathcal{L}(-\lambda)u_-=\mathcal{L}(-\lambda) A_- u$.  The conclusion then follows from applying Proposition~\ref{prop:limit-reg} to $u$.   
 \end{proof}

\section{Proof of the main theorem}\label{sec:proof}
We now elaborate on the construction and regularity of nonlinear
Helmholtz eigenfunction $u$, whose asymptotic behavior is the subject
of the main theorem.  We begin by discussing (linear) generalized
eigenfunctions.

Given $f \in H^k(\partial M)$, we let $u_0 = P(\lambda) f$ and decompose using Proposition~\ref{prop:L2} into 
$$
u_0 = u_- + u_+ , \qquad u_\pm = A_\pm u_0 \in \X_\pm^{s+2, -1/2 - \delta; k, 0}(M)
$$
where $A_\pm$ are as in \eqref{eq:main commutator}. According to Proposition~\ref{prop:limit-reg},
we have 
\begin{equation}
  \label{eq:free ass}
u_0 = u_- + u_+   = r^{-(n-1)/2} \Big( e^{-i\lambda r} f(y) + e^{+i\lambda r}
b_0(y) + O_{H^{k - 2}}(r^{-\epsilon}) \Big),
\end{equation}
where $b_0$ is in $H^k(\partial M)$. 

To address the nonlinear problem, following \cite{ghsz20}, we obtain a nonlinear Helmholtz eigenfunction $u$ satisfying
\begin{equation}
  \begin{gathered}
    u = u_- + w, \quad u \mbox{ solves \eqref{eq:equation}}, \\
    u_- \in \X_-^{s+2, -1/2 -\delta; 1, k-1}, \quad w \in
\X_+^{s+2, -1/2 -\delta; 1, k-1}, \quad s \in \NN,\\
    w = u_+ + R(\lambda + i0)N[u_- + w].
  \end{gathered}
\label{upmw} \end{equation}
Moreover, if the nonlinearity $N$ involves products of
degree $p$, then, as described in detail in Section 4.2 of
\cite{ghsz20}, 
$$
N[u_- + w] \in H_+^{s, \ell'; 1, k-1} 
$$
provided
\begin{equation}
\ell' \le 
\frac{(p-1)(n-1)}{2} - \frac 32 - p\delta. \label{eq:4}
\end{equation}
The contraction mapping argument which produces this $w$ requires that
$\ell' = 1/2 - \delta$ for the same $\delta$ appearing in
\eqref{eq:decay function} and \eqref{eq:eps-cond}, 
whence the bound for $p$ in \eqref{pkcond}, which in fact allows for
$\ell' = 1/2 + \delta$ for $\delta >0$ sufficiently small to satisfy
all the above conditions.  Thus $F = N[u_- + w]$ satisfies the
  assumptions of Proposition \ref{prop:reslimits}.  We conclude that 
\begin{equation}
w - u_+ =  r^{-(n-1)/2} e^{i \lambda r} \Big(b_1(y) + O_{H^{k -2}}(r^{-\epsilon}) \Big) , \quad b_1 \in H^k(\partial M), 
\label{w}\end{equation}
using Proposition~\ref{prop:reg} for the regularity of $b_1$. 
Combining $u = u_0 + (w - u_+)$ using \eqref{eq:free ass} and \eqref{w} proves the asymptotic behavior stated in the main theorem,
with $b = b_0 + b_1 \in H^k(\partial M)$.

Now we assume we are given $k \in \mathbb{N}$, $k > (n -1)/2$, and $f \in
H^k(\partial M)$ with $\| f
\|_{H^k(\partial M)} < c$, as in the statement of Theorem
\ref{thm:main2}. In addition, we suppose 
$$
(p-1)(n -1)/2 > 3
$$
and that $N[u]$ only involves derivatives of $u$ and $\overline{u}$ up to order one. 
Then by \eqref{eq:4} we can take $\ell' = 3/2 + \delta$, which is to say we obtain $w$ with
$$
N[u_- + w] \in H_+^{s+1, 3/2 + \delta; 1, k-1}  \subset H_+^{s, 1/2 +
  \delta; 2, k - 1} \subset H_+^{s, 1/2 +
  \delta; 1, k}.
$$
The first containment is because we can exchange one order of scattering differentiability plus one order of spatial decay, and gain one order of module regularity. 
Then applying \eqref{eq:res map mod reg}, we see that $w-u_+$ is in the better space $\X_+^{s+2, -1/2 -\delta; 1, k}$, that is, one additional order of small module regularity compared to \eqref{upmw}. 
Applying Propositions \ref{prop:reslimits} and \ref{prop:reg}, we have
\begin{equation*}
w - u_+ =  r^{-(n-1)/2} e^{i \lambda r} \Big(b_1(y) + O_{H^{k -
    1}}(r^{-\epsilon}) \Big) , \quad b_1 \in H^{k + 1}(\partial M), 
\end{equation*}
Thus under the stronger assumption on $p$, the decomposition $b = b_0
+ b_1$ holds with $b_0 = S_{\mathrm{lin}}(\lambda)f \in H^k$ and $b_1 \in H^{k + 1}$.  If $f \in
H^{k + j}(\partial M)$ (together with the $H^k$ smallness assumption
on $f$) then $u_\pm \in H_\pm^{s, 1/2 - \delta; k + j, 0}$ and by a
bootstrap argument we get $b_1\in H^{k + j+1}(\partial M)$.  In
particular, since $S_{\mathrm{lin}}(\lambda)$ is an FIO of order zero \cite{mezw96}, if $f$ is in $C^\infty(\partial M)$ then $b= S_{\mathrm{lin}}(\lambda)f + b_1$ is also in $C^\infty(\partial M)$.

Uniqueness follows from the same considerations as in \cite{ghsz20}.
Namely, given $u_- \in \mathcal{X}_-^{s+2, -1/2 - \delta; 1, k-1}$, as the
function $w$ above is produced using a contraction mapping on $\mathcal{X}_-^{s+2, -1/2 - \delta; 1, k-1}$ , $w = u - u_-$ is
uniquely determined in a small ball in this space.


\begin{thebibliography}{0}

\bibitem[GHSZ20]{ghsz20} J. Gell-Redman, A. Hassell, J. Shapiro, and J. Zhang. Existence and asymptotics of Nonlinear Helmholtz eigenfunctions. \textit{SIAM J. Math. Anal.} 52(6) (2020), 6180--6221

\bibitem[HV99]{HV99}  A. Hassell, A.  Vasy, The spectral projections and the resolvent for scattering metrics. J. Anal. Math. 79 (1999), 241--298. 

\bibitem[HMV04]{HMV04}  A. Hassell, R. B.  Melrose, A. Vasy, Spectral and scattering theory for symbolic potentials of order zero,  Adv. Math. 181 (2004), no. 1, 1--87.

\bibitem[Hvol4]{Hvol4} L. H\"ormander, The analysis of linear partial differential operators IV, Springer, 1985. 

\bibitem[Ik60]{Ik60}  T. Ikebe, Eigenfunction expansions associated with the Schr\"odinger operators and their applications to scattering theory. Arch. Rational Mech. Anal. 5 (1960), 1--34 (1960). 

\bibitem[Is82]{Is82}   H. Isozaki, On the generalized Fourier transforms associated with Schr\"odinger operators with long-range perturbations. J. Reine Angew. Math. 337 (1982), 18--67.

\bibitem[IK84]{IK84} H.  Isozaki, H.  Kitada, Microlocal resolvent estimates for 2-body Schr\"odinger operators. J. Funct. Anal. 57 (1984), no. 3, 270--300.

\bibitem[Me94]{me94} R. Melrose. Spectral and scattering theory for the Laplacian on asymptotically Euclidean spaces. In \textit{Spectral and scattering theory} (Sanda, 1992), volume 161 of \textit{Lecture notes in Pure and Appl. Math.}, 85--120, Dekker, New York, 1994.

\bibitem[MeZw96]{mezw96} R. Melrose and M. Zworski. Scattering metrics and geodesic flow at infinity. \textit{Invent. Math.} 123(3), (1996), 389--436.

\bibitem[Va98]{Va98} A.  Vasy, Geometric scattering theory for long-range potentials and metrics. Internat. Math. Res. Notices 1998, no. 6, 285--315. 

\bibitem[Va13]{Va13}  A. Vasy, Microlocal analysis of asymptotically hyperbolic and Kerr-de Sitter spaces (with an appendix by Semyon Dyatlov). Invent. Math. 194 (2013), no. 2, 381--513. 

\bibitem[Va18]{va18}
A.~Vasy.
\newblock A minicourse on microlocal analysis for wave propagation.
\newblock In {\em Asymptotic analysis in general relativity}, volume 443 of
  {\em London Math. Soc. Lecture Note Ser.}, pages 219--374. Cambridge Univ.
  Press, Cambridge, 2018.

\end{thebibliography}
\end{document}